\documentclass{article}

\usepackage[margin=1in]{geometry}

\usepackage[all]{xy}
\usepackage{amsmath,amssymb,amsthm}
\usepackage{chngcntr}
\usepackage{enumerate}
\usepackage{graphicx}
\usepackage{mathrsfs}
\usepackage{mathtools}
\usepackage{paralist}

\theoremstyle{definition}
\newtheorem{theorem}{Theorem}
\newtheorem*{theorem*}{Theorem}
\newtheorem{prop}[theorem]{Proposition}
\newtheorem*{prop*}{Proposition}
\newtheorem{lem}[theorem]{Lemma}
\newtheorem*{lem*}{Lemma}
\newtheorem{cor}[theorem]{Corollary}
\newtheorem*{cor*}{Corollary}
\newtheorem{definition}[theorem]{Definition}
\newtheorem*{definition*}{Definition}
\newtheorem{rem}[theorem]{Remark}
\newtheorem*{rem*}{Remark}
\newtheorem{notation}{Notation}    
\newtheorem*{notation*}{Notation}    
\numberwithin{theorem}{subsection} 	 
\numberwithin{equation}{section}

\newcommand*{\plim}[1]{\mathop{\varprojlim}\limits_{#1}}	
		
\newcommand*{\set}[2]{\left\{#1\ \middle|\ #2\right\}}		
\newcommand*{\Ker}[0]{\text{Ker}}						
						
\newcommand*{\pro}[0]{\text{pro-}}						
\newcommand*{\ab}[0]{\text{ab}}							
\newcommand*{\nil}[0]{\text{nil}}							
\newcommand*{\sep}[0]{\text{sep}}						
\newcommand*{\op}[0]{\text{op}}

\newcommand*{\aru}[1]{\ar@{}[u]|*=0[@]{\rotatebox{90}{$#1$}}}

\begin{document}

\title{The geometrically  $m$-step solvable Grothendieck conjecture\\ for genus $0$ curves over finitely generated fields}
\author{Naganori Yamaguchi}
\date{}
\maketitle

\begin{abstract}
In this paper, we present some partial results for the geometrically   $m$-step solvable Grothendieck conjecture in anabelian geometry. Among other things, we prove the  geometrically  $3$-step  solvable Grothendieck conjecture for genus $0$ curves over fields finitely generated over the prime field of arbitrary characteristic.

\end{abstract}

%\tableofcontents

\addcontentsline{toc}{section}{Introduction}
\section*{Introduction}
\hspace{\parindent} In this paper, a curve over a field $k$ is defined as a one-dimensional scheme geometrically connected, separated and of finite type over $k$. \par
Let $i=1,2$. Let $X$ (resp. $X_{i}$) be a smooth proper curve over $k$. Let  $E$ (resp. $E_{i}$) be  a closed subscheme of $X$ (resp. $X_{i}$) which is finite \'{e}tale over $k$. Set $U:=X-E$ (resp. $U_{i}:=X_{i}-E_{i}$). Let  $\pi_1(U)$  denote the \'{e}tale fundamental group of $U$. Let $\pi_1^{\text{tame}}(U)$ be the tame fundamental group of $U$. \par
Let $p:=\text{ch}(k)$. Let $g(U)$ be  the genus of $U$ ($:=$ the genus of $X$). Let  $r(U):=|E_{\overline{k}}|$. We say that $U$ is affine when $r(U)>0$, and that  $U$ is hyperbolic when $2-2g(U)-r(U)<0$. \par
The Grothendieck conjecture (in anabelian geometry) asks: if a $G_k$-isomorphism $\pi_1(U_{1})\cong\pi_1(U_{2})$ exists, does a $k$-isomorphism $U_{1}\cong U_{2}$ exist? About this conjecture, we already have many results (\cite{Na1990-411}, \cite{Ta1997}, \cite{Mo1999}, \cite{St2002}, \cite{St2002P}, etc....). For example, we have the following two results.

\begin{theorem}(\cite{Mo1999}Theorem A)\label{thmt1}
Assume that $k$ is  finitely generated over $\mathbb{Q}$ and  $U_{1}$ is  hyperbolic. Then the following holds.
\begin{equation*}
\pi_1(U_{1})\underset{G_k}{\cong} \pi_1(U_{2})\Longleftrightarrow U_{1}\underset{k}{\cong}U_{2}
\end{equation*}
\end{theorem}

\begin{theorem}(\cite{St2002}Theorem 1)\label{thmt2}
Assume that  $k$ is   finitely generated over a finite field  and $U_{1}$ is non-isotrivial affine hyperbolic. Then the following holds.
\begin{equation*}
\pi_1^{\text{tame}}(U_{1})\underset{G_k}{\cong} \pi_1^{\text{tame}}(U_{2})\Longleftrightarrow
U_{1}(n_{1})\underset{k}{\cong}U_{2}(n_{2}) \text{ for some }n_{1}, n_{2}\in \mathbb{N}\cup \left\{0\right\}
\end{equation*}
Here, $U_{i}(n_{i})$ is  the $n_{i}$-th Frobenius twist of $U_{i}$ for $i=1,2$  and  $\mathbb{N}$ is  the set of all positive integers. 
\end{theorem}

Let $m\in\mathbb{N}$.  For any topological  group $G$, we define   $G^{[0]}:=G$,  $G^{[m]}:=[G^{[m-1]},G^{[m-1]}]$, and   $G^{m}:=G/G^{[m]}$. Let $\pi_{1}^{m}(U_{k^\text{sep}}):=\pi(U_{k^\text{sep}})^{m}$ and $\pi_1^{(m)}(U):=\pi_1(U)/\pi_1(U_{k^\text{sep}})^{[m]}$.  We write  $\pi_1^{\text{pro-}p'}(U_{k^\text{sep}})$ for the maximal pro-prime-to-$p$ quotient of $\pi_1(U_{k^{\text{sep}}})$. (We define $\pi_1^{\text{pro-}0'}(U_{k^\text{sep}})):=\pi_1(U_{k^\text{sep}})$.)  Let  $\pi^{\text{pro-}p',m}(U_{k^\text{sep}}):=\pi^{\text{pro-}p'}(U_{k^\text{sep}})^{m}$ and $\pi_1^{(\text{pro-}p',m)}(U):=\pi_1(U)/\text{Ker}(\pi_{1}(U_{k^\text{sep}})\rightarrow \pi_1^{\text{pro-}p',m}(U_{k^\text{sep}}))$.  \par
We consider the following question.
\begin{equation*}
\pi_1^{(m)}(U_{1})\underset{G_k}{\cong} \pi_1^{(m)}(U_{2})\Longrightarrow U_{1}\underset{k}{\cong}U_{2}\text{ ?}
\end{equation*}
This question is referred to as the (geometrically) $m$-step solvable Grothendieck conjecture. For the $m$-step solvable Grothendieck conjecture, we have some previous results,  such as \cite{Na1990-405}Theorem B ($m=2$, $k$ a number field  satisfying certain conditions, $U_{1}$ hyperbolic, $g(U_{1})=0$) and \cite{Mo1999}Theorem A$'$ ($m\geq 5$, $\ell$ a prime,  $k$ a sub-$\ell$-adic field, $U_{1}$ hyperbolic, $g(U_{1})$ general).\par
In this paper, we present some new results for the $m$-step solvable Grothendieck conjecture for genus $0$ hyperbolic curves over finitely generated field  with a certain non-isotriviality assumption in positive characteristic.
More precisely, when $p\neq0$, we consider the following assumption for $U_{1}$ (cf. Theorem \ref{2.4.1}).

\begin{equation*}
(*):\text{For each }S'\subset E_{1,\overline{k}}\ \text{with}\ |S'|=4,\text{ the curve }X_{1,\overline{k}}-S'\ \text{is\ not\  isotrivial.}
\end{equation*}
It is clear that $U_{1}$ is not isotrivial if $U_{1}$ satisfies ($*$) and $r(U_{1})\geq 4$. (We will explain later about this condition.) The following theorem is the main result that we  prove in this paper.

\begin{theorem}\label{intr0}
Let  $m\geq 3$. Assume that $k$ is finitely generated over the prime field and  $U_{1}$ is  a genus 0 hyperbolic curves over $k$. If, moreover  $p>0$, we assume the condition (*). Then the following holds.
\[
\pi_1^{(\text{pro-}p',m)}(U_{1})\underset{G_k}{\cong} \pi_1^{(\text{pro-}p',m)}(U_{2})\Longrightarrow
\begin{cases}
U_{1}\underset{k}{\cong}U_{2} & (p=0) \\
U_{1}(n_{1})\underset{k}{\cong} U_{2}(n_{2})\text{ for some }n_{1},n_{2}\in \mathbb{N}\cup \left\{0\right\} & (p>0)
\end{cases}
\]
\end{theorem}
\noindent Next, we give an outline of the proof of Theorem 0.0.3. The flow of this proof is based on  \cite{Na1990-411} in many parts (in particular, when $p=0$). However, our proof differs from that of   \cite{Na1990-411} in the following two points.
\begin{itemize}
\item[(P.1)]  We consider the case that $k$ is a field finitely generated over the prime field, while, in  \cite{Na1990-411}, Nakamura only considers the case that $k$ is a number field. In particular, we deal with the case that $p>0$. In this case, the proof developed in \cite{Na1990-411} does not work and we need to develop a  new argument (c.f. subsection \ref{chp}). 
\item[(P.2)]  Most of our arguments start with $\pi^{(\text{pro-}p',m)}(U)$ because we deal with  the $m$-step solvable Grothendieck conjecture, while,  in  \cite{Na1990-411}, most of  Nakamura's arguments  start with the full fundamental group $\pi(U)$. Thus, we need to prove various  results for $\pi^{(\text{pro-}p',m)}(U)$ that are already known for $\pi(U)$. However, since many of the known proofs for $\pi(U)$ cannot be used as is, we need to  develop a new argument for the proof for $\pi^{(\text{pro-}p',m)}(U)$ (c.f. subsection \ref{subsection1.1}, subsection \ref{subsection1.2}, and subsection \ref{subsection1.4}).
\end{itemize}
 
 \noindent  Let us  explain the two steps of the proof by focusing the difference between \cite{Na1990-411}  and our proof (in particular, (P.1) and (P.2)).\ \\

\noindent \underline{Step 1} (contents in section 1)\par
In this step, we show the group-theoretical reconstruction of decomposition  groups of $\Pi^{(m)}(U)$ from $\Pi^{(m+2)}(U)$ (contents in section 1). This step is divided into two parts.\par
 First (contents in subsection 1.2, subsection  1.3, and subsection  1.4),  we show the group-theoretical reconstruction of inertia  groups of  $\pi_{1}^{(\text{pro-}p',m)}(U)$ from $\pi_{1}^{(\text{pro-}p',m+2)}(U)$.  We consider the notion of   maximal cyclic subgroups of cyclotomic type. (Roughly speaking, this is a cyclic subgroup of $\pi_{1}^{\text{pro-}p',m}(U_{k^{\text{sep}}})$ on which  the Galois group acts via the cyclotomic character.)   The  maximal cyclic subgroups of cyclotomic type are  first defined in  \cite{Na1990-411} Definition 3.3 in the case of the full fundamental group.  However, the definition  is not sufficient for our proof and  we need to redefine it, taking (P.1) and (P.2) into careful consideration.  By using the  maximal cyclic subgroups of cyclotomic type, we obtain  the group-theoretical reconstruction of inertia  groups of  $\pi_{1}^{(\text{pro-}p',m)}(U)$ from $\pi_{1}^{(\text{pro-}p',m+2)}(U)$ (Proposition \ref{1.4.2}). \par
 Next (contents in subsection 1.1), we show that  $D_{y}=N_{\pi_{1}^{(\text{pro-}p',m)}(U)}(I_{y})$, where  $D_{y}:=D_{y,\pi_{1}^{(\text{pro-}p',m)}(U)}$ and $ I_{y}:=I_{y,\pi_{1}^{(\text{pro-}p',m)}(U)}$ stand for the decomposition group and the  inertia group at a cusp $y$ in $\pi_{1}^{(\text{pro-}p',m)}(U)$, respectively. The main pat of this part is to prove the following proposition:

\begin{prop}[Proposition \ref{1.1.9}]\label{intr1}
Let $C$ be a  full class of finite groups which contains a non-identity group. Let  $\mathcal{F}$ be a free pro-$C$ group  and  $X \subset  \mathcal{F}$  a set of free generators of $\mathcal{F}$.  If  $m\geq 2$,  then $Z_{\mathcal{F}^m}(x^n) = \langle x\rangle$ holds for every $x\in X$ and every $n\in \mathbb{Z}-\left\{0\right\}$. 
\end{prop}
 \noindent If we replace $\mathcal{F}^{m}$ in Proposition \ref{intr1}  with $\mathcal{F}$, the assertion  is well-known. However, the following proposition is a new result, as far as the author knows, which may be of some independent interest.  (This is the problem that comes from (P.2)). To prove  Proposition \ref{intr1}, we establish and  use the   Blanchfield-Lyndon theory for free pro-$C$ groups, which is given in Appendix (see Theorem A.2).  By Proposition \ref{intr1}, we obtain that  $D_{y}=N_{\pi_{1}^{(\text{pro-}p',m)}(U)}(I_{y})$ when $r\geq  2$ and  $\pi_{1}^{\text{pro-}p',m}(U_{k^\text{sep}})$ is not abelian (see Proposition \ref{1.2.3}).  By using these results, we  get   the group-theoretical reconstruction of decomposition  groups of $\pi_{1}^{(\text{pro-}p',m)}(U)$ from $\pi_{1}^{(\text{pro-}p',m+2)}(U)$ when $m\geq 2$ and $r_{1}\geq 2$ (see Corollary \ref{1.4.5}). When $m=1$,   $D_{y}=N_{\pi_{1}^{(\text{pro-}p',m)}(U)}(I_{y})$ is not correct. To get the result in the case that $m=1$ and $r_{1}\geq 3$, we must show that  $D_{y,\pi_{1}^{(\text{nil},m)}(U)}=N_{\pi_{1}^{(\text{nil},m)}(U)}(I_{y,\pi_{1}^{\text{nil},m}(U_{k^\text{sep}})})$, where   $\pi_{1}^{\nil,m}(U_{k^{\text{sep}}})$ stands for the maximal nilpotent $m$-step solvable quotient of $\pi_{1}(U_{k^{\text{sep}}})$ and $\pi_{1}^{(\nil,m)}(U):= \pi_{1}(U)/\text{Ker}(\pi_{1}(U_{k^{\text{sep}}})\rightarrow \pi_{1}^{\nil,m}(U_{k^{\text{sep}}}))$ (see Proposition \ref{1.2.7}). \ \\

\noindent \underline{Step 2} (contents in section 2)\par
In this step, we show Theorem \ref{intr0} by using the result of Step 1. Since we face the problem (P.1)  frequently in this step, we mainly explain  the case that $p>0$. This step is divided by the  following three parts. \par 
First, (contents in subsection 2.1, subsection 2.2, and subsection 2.3), we show Theorem \ref{intr0} when $U_{i}=\mathbb{P}^{1}_{k}-\left\{0,\infty,1,\lambda_{i}\right\}$, where   $\lambda_{i}\in k^{\times}-\{1\}$.  First, we reconstruct $\langle\lambda\rangle$ and $\langle1-\lambda\rangle$ (resp. $\langle\gamma\rangle^{p^{u}}$ and  $\langle1-\gamma\rangle^{p^{v}}$ for some $u,v\in\mathbb{Z}$)  from  $\pi_{1}^{(\text{pro-}p',3)}(\mathbb{P}^{1}_{k}-\{0,\infty,1,\lambda\})$  when $p=0$ (resp. when $p>0$)  (see Proposition \ref{2.1.3} and Corollary \ref{2.1.6}).  This  is done by using the results of Step 1 and the rigidity invariant  (see Definition \ref{2.1.1}).  We reconstruct $\lambda\in k^{\times}-\left\{1\right\}$ from  $\langle\lambda\rangle$ and $\langle1-\lambda\rangle$ when $p=0$ and $\lambda\notin\left\{\rho,\rho^{-1}\right\}$ (or when $p>0$ and $\lambda$ is not a torsion element),  where $\rho\in \overline{k}$ is a primitive $6$-th root of unity (see Lemma \ref{2.2.1}).  When $p=0$, this result is basically shown in \cite{Na1990-411}Theorem 4.4. (The only difference between   \cite{Na1990-411}Theorem 4.4 and   Lemma \ref{2.2.1}  lies in what comes from  (P.1).)  For the analogy of this reconstruction in the case that $p>0$,  we need  the assumption that $\lambda\in k^{\times}$ is not a torsion element.  This assumption is the cause of the assumption $(*)$ in Theorem 0.0.3.  By using this result and Step 1,  Theorem \ref{intr0} holds when $U_{i}=\mathbb{P}^{1}_{k}-\{0,1,\infty,\lambda_{i}\}$, where $\lambda_{1}, \lambda_{2}\in k-\{0,1\}$ (resp.  $\lambda_{1}, \lambda_{2}\in k-k_{0}$ when $p=0$ (resp. $p>0$). \par
 Next,  (contents in subsection 2.2 and  subsection 2.3), we show Theorem \ref{intr0} when $U_{i}=\mathbb{P}^{1}_{k}-E_{i}$ with  $E_{i}\subset \mathbb{P}^{1}_{k}(k)$. Note that we have that  $|E_{1}|\geq 3$ by the assumption of the hyperbolicity of $U_{1}$. When $|E_{1}|\geq 3$, the assertion is easy. When $|E_{1}|=4$, the assertion is already shown. Hence we explain the case that $|E_{1}|\geq 5$.    When $p=0$,  we obtain  that Theorem \ref{intr0} for $|E_{1}|\geq 5$ follows from Theorem \ref{intr0} for $|E_{1}|=4$ (see Proposition \ref{2.2.4}) by dividing $\pi_{1}^{(\text{pro-}p',m)}(\mathbb{P}^{1}_{k}-E_{1})$ by inertia groups over $E_{1}-S'$ for various   subsets   $S_{1}'$ of $E_{1}$ satisfying $|S_{1}'|=4$.  When $p>0$, the analogy of this is more difficult, since  we have to consider the Frobenius twists which appear in the following way: Suppose we are given an isomorphism $\alpha$ between $\pi_{1}^{(\text{pro-}p',m)}(\mathbb{P}^{1}_{k}-E_{i})$ for $i=1,2$.  By using Step 1 and  $\alpha$, we get the subset $S'_{2}$ of $E_{2}$ with $|S_{2}'|=4$. There exist  isomorphisms of $\mathbb{P}_{k}^{1}$ such that $\mathbb{P}_{k}^{1}-S_{i}'$ maps $\mathbb{P}_{k}^{1}-\{0,\infty,1,\lambda_{i}\}$ for some $\lambda_{i}\in k^{\times}-\{1\}$ ($i=1,2$). Then we get $n\in\mathbb{N}$ with $\lambda_{2}=\lambda_{1}^{p^{n}}$ by Lemma \ref{2.3.5}. This $n$ depends on the isomorphism of $\mathbb{P}_{k}^{1}$  and $S_{1}'$.  We have to unify the twists, which is the most difficult part of considering the case that $p>0$.    For this,  we obtain several relations of the twists by  taking  another isomorphism of $\mathbb{P}_{k}^{1}$ (see the proof of Proposition \ref{2.3.7}) and Lemma \ref{2.3.6}.  Thus, we  obtain  Theorem \ref{intr0} when $U_{1}$, $U_{2}$ are  punctured projective lines (satisfying $(*)$).  \par 
   Finally (contents in subsection 2.4),  by using Galois  descent theory, we reduce the general case to above results. Thus, the proof of Theorem \ref{intr0} is done.

\section*{Acknowledgments}
I gratefully acknowledge Professor Akio Tamagawa (Research Institute for Mathematical Sciences, Kyoto University) for his advice and his help.

%%%%%%%%%%%%%%%%%%%%%%%%%%%%%%%%%%%%%%%%%%%%%%%%%%%%%%%%%%%%%%%%%%%%%%%%%%%%%%%%%%%%%%%%%%%%%
%%%%%%%%%%%%%%%%%%%%%%%%%%%%%%%%%%%%%%%%%%%%%%%%%%%%%%%%%%%%%%%%%%%%%%%%%%%%%%%%%%%%%%%%%%%%%
%%%%%%%%%%%%%%%%%%%%%%%%%%%%%%%%%%%%%%%%%%%%%%%%%%%%%%%%%%%%%%%%%%%%%%%%%%%%%%%%%%%%%%%%%%%%%

\section{Reconstruction of inertia groups and decomposition groups at cusps}\label{section1}

\hspace{\parindent}In this section, we reconstruct inertia groups and decomposition groups at cusps of  the maximal $m$-step solvable quotient of fundamental groups of curves, group-theoretically. In subsection \ref{subsection1.1}, we show some properties of  the maximal $m$-step solvable quotients of free pro-$C$ groups. Here the most important ingredient is pro-$C$ Blanchfield-Lyndon theory. (See the proof of Lemma \ref{1.1.6}.) In subsection \ref{subsection1.2}, we introduce some   notations and show basic properties of inertia groups and decomposition groups at cusps. In particular, we investigate the intersection of two inertia groups. In subsection \ref{subsection1.3}, we investigate the weight filtration of the abelianized fundamental group of curves. In subsection \ref{subsection1.4},   we reconstruct inertia groups by using  the notion of  maximal cyclic group of cyclotomic type. \par
This section mainly refers to sections 2 and  3 of  \cite{Na1990-411}.

%%%%%%%%%%%%%%%%%%%%%%%%%%%%%%%%%%%%%%%%%%%%%%%%%%%%%%%%%%%%%%%%%%%%%%%%%%%%%%%%%%%%%%%%%%%%%

\subsection{The maximal $m$-step solvable quotients of  free pro-$C$ groups}\label{subsection1.1}
\hspace{\parindent}In this subsection, we give the definition of free pro-$C$ groups and some results on them. The main object in this subsection is the centralizer of the group generated by an element of  a basis of  a free pro-$C$ group. In later subsections, this group coincides with inertia groups. 

For   a topological group  $G$ and $m\in \mathbb{N}$, we set $G^{[0]}:=G$,\ $G^{[m]}:=[G^{[m-1]},G^{[m-1]}]$, and  $G^{m}:=G/G^{[m]}$  (For topological groups, the term ``generated''  means ``topologically generated'' in this paper.) For a class of finite groups $C$,  we say that $C$ is almost full if $C$ is closed under taking quotients, subgroups and direct products. We say that $C$ is  full if $C$ is almost full and closed under taking extensions.    Moreover, we introduce the following notations:  $\text{prime}(C)$:=$\set{l\in\mathbb{N}}{l\text{ is a prime satisfying }\mathbb{Z}/l\mathbb{Z}\in C.}$,  $\mathbb{N}(C):=\set{n\in\mathbb{N}}{\mathbb{Z}/n\mathbb{Z}\in C}$, and  $\mathbb{Z}_{C}:=\underset{l\in \text{prime}(C)}\prod\mathbb{Z}_{l}$. We denote the maximal pro-$C$ quotient of $G$ by $G^{C}$. (If $C$ coincides with the class of all $\ell $-groups, we write $G^{\text{pro-}\ell}$ instead of $G^{C}$.)  For an  almost full class  $C$ of finite groups and  for each set $X$, a free pro-$C$ group with basis $X$ exists and is uniquely determined up to an isomorphism (\cite{Fa2008}Lemma 17.4.6). We denote it by $ F^{C}(X)$. (When $X=\emptyset$, $F^{C}(X)=\left\{1\right\}$.)

\begin{notation}\label{notation1}
Throughout this paper, we fix the following notations
\begin{enumerate}[(i)]
\item $m\in\mathbb{N}$.
\item $C$ is a full class of finite groups which contains a non-identity group. 
\item $\mathcal{F}$ is a free pro-$C$ group and $X$ is a free basis of $\mathcal{F}$.
\end{enumerate}
\end{notation}

\begin{prop}\label{1.1.5}
Let $x\in X$ and $\alpha\in \mathbb{Z}_{C}-\left\{0\right\}$. Then $Z_{\mathcal{F}^m}(x^\alpha) \ \subset\  \langle x\rangle\cdot\mathcal{F}^{[m-1]}/\mathcal{F}^{[m]}$ holds. Here, $Z_{\mathcal{F}^m}(x^\alpha) $ stands for the centralizer of $x^\alpha $ in $\mathcal{F}^m$.
\end{prop}

\begin{proof}(cf. \cite{Na1994} Lemma 2.1.2) 
If $m=1$, the assertion is clear because $\mathcal{F}^{[0]}/\mathcal{F}^{[1]}=\mathcal{F}^{1}$. Thus, we may assume that $m\geq 2$. First, we consider the case $|X|<\infty$.\par
Let $y\in Z_{\mathcal{F}^m}(x^\alpha) $. For each $N\overset{\op}\lhd\mathcal{F}$ such that $(F/N)^{[m-1]}=\left\{1\right\}  $ (or, equivalently, $N\supset \mathcal{F}^{[m-1]}$), consider the natural surjection $\rho_N:\mathcal{F}\rightarrow \mathcal{F}/N$. Then 
\begin{eqnarray*}
\rho_N(y)\in \langle \rho_N(x)\rangle \text{ for all }N&\Rightarrow&\ y\in \langle x \rangle\cdot N/\mathcal{F}^{[m]} \text{ for all }N\\
&\Rightarrow& y\in \bigcap_N \left(\langle x\rangle\cdot N/\mathcal{F}^{[m]}\right) =\langle x\rangle\cdot \bigcap_N \left(N/\mathcal{F}^{[m]}\right) =\langle x\rangle\cdot \mathcal{F}^{[m-1]}/\mathcal{F}^{[m]}
\end{eqnarray*}
Thus, it suffices to prove that $\rho_N(y)\in\langle\rho_N(x)\rangle$ for all $N$. We fix any $N$ and write $G:=\mathcal{F}/N $and $\rho:=\rho_N$.\par
Write $\alpha=(\alpha_\ell)_{\ell\in \text{prime}(C)}\in \mathbb{Z}_C-\left\{0\right\}$ and fix $\ell \in \text{prime}(C)$ such that $\alpha_\ell\neq 0$. We also fix a sufficiently large $s$ and an injection $G\hookrightarrow GL_s(\mathbb{Z}_\ell)$ (say, arising from a permutation representation). Via this injection, we regard $G$ as a subgroup of $GL_s(\mathbb{Z}_\ell)$. Further, we set
\[
  G' := \set{\ \left(
    \begin{array}{cc}
      A & B \\
      0 & C 
    \end{array}
  \right)\in GL_{2s}(\mathbb{Z}_{\ell})}{A\in G,\ C\in \langle \rho(x)\rangle}\text{.}
\]
By consider  the diagonal component $\left(
\begin{array}{cc}
      A & 0 \\
      0 & C 
    \end{array}
  \right)$
and the unipotent component
$ \left(
    \begin{array}{cc}
      1 & B \\
      0 & 1 
    \end{array}
  \right)$,
we obtain the following exact sequence.
\begin{equation}\label{G'の完全系列}
1\rightarrow \left\{\left(\begin{array}{cc}1 & B \\  0& 1 \end{array}\right) \middle| B\in M_s(\mathbb{Z}_\ell)\right\}\rightarrow G'\rightarrow G\times \langle \rho(x)\rangle\rightarrow 1
\end{equation}
The right term of (\ref{G'の完全系列}) is a pro-$C^{m-1}$ group and the left term is a pro-$C^{1}$ group because $\ell\in\text{prime}(C)$. Thus, $G'$ is a pro-$C^m$ group. \par
We set $x_1:=x$ and $X:=\left\{x_1,x_2,\cdots,x_r\right\}$.  We define a map $X\rightarrow G'$ by
\begin{equation*}
x\mapsto \left(
    \begin{array}{cc}
      \rho(x) & \rho(x) \\
      0 & \rho(x) 
    \end{array}
  \right)
,x_i\mapsto \left(
    \begin{array}{cc}
      \rho(x_i) & 0 \\
      0 & 1 
    \end{array}
  \right)\ (2\leq i\leq r).
\end{equation*}
This map extends to $\psi:\mathcal{F}^m\rightarrow G'$. Set $g:=|G|$, then we get  
\begin{equation*}
\psi(x^{\alpha g})=
\left(
    \begin{array}{cc}
      \rho(x) & \rho(x) \\
      0 & \rho(x) 
    \end{array}
  \right)^{\alpha g}
=\left(
    \begin{array}{cc}
      1 & g\alpha_{\ell} \\
      0 & 1 
    \end{array}
  \right)\ \ \ \ 
\end{equation*} 
Write $\psi(y)=
\left(
    \begin{array}{cc}
      \rho(y) & B \\
      0 & C 
    \end{array}
  \right)$. 
As $y\in Z_{\mathcal{F}^m}(x^{\alpha})$, $y$ and $x^{\alpha g}$ are commutative. Therefore, the following two products in $G'$ are equal.
\begin{eqnarray*}
\psi(y)\psi(x^{\alpha g})&=&
\left(
    \begin{array}{cc}
      \rho(y) & B \\
      0 & C 
    \end{array}
  \right)
\left(
    \begin{array}{cc}
      1 & g\alpha_{\ell} \\
      0 & 1 
    \end{array}
  \right)
=
\left(
    \begin{array}{cc}
      \rho(y) & g\alpha_{\ell}\rho(y)+B\\
      0 & C 
    \end{array}
  \right)
\\
\psi(x^{\alpha g})\psi(y)&=&
\left(
    \begin{array}{cc}
      1 & g\alpha_{\ell} \\
      0 & 1 
    \end{array}
  \right)
\left(
    \begin{array}{cc}
      \rho(y) & B \\
      0 & C 
    \end{array}
  \right)
=
\left(
    \begin{array}{cc}
      \rho(y) & g\alpha_{\ell}C+B\\
      0 & C 
    \end{array}
  \right)
\end{eqnarray*}
By comparing the (1,2) components,  we get $\rho(y)=C\in \langle\rho(x)\rangle$. Thus, the assertion holds if $|X|<\infty$.\par
Finally, we consider the case that  $|X|=\infty$. Let $\left\{X_i\right\}_{i\in I}$ be the set of all finite subsets of $X$ that contain $x$. 
Let $u\in\mathbb{N}\cup\{0\}$ and set   $C^u:=\set{M\in C}{M^{[u]}=\left\{1\right\}}$. We remark that $C^{u}$ is also almost full and $F^C(X)^u\simeq F^{C^u}(X)$.   For each pair $i,j\in I$ that satisfy $X_i\subset X_j$, define the map $\tau^{u}_{ij}:F^{C^{u}}(X_j)\twoheadrightarrow  F^{C^{u}}(X_i)$ by $x\mapsto x$ for  $x\in X_i$ and $x\mapsto 1$ for  $x\in X_j-X_i$.  Also, for each $i\in I$, define the map $\tau^{u}_i:(\mathcal{F}^{u}=)F^{C^{u}}(X)\twoheadrightarrow  F^{C^{u}}(X_i)$ by $x\mapsto x$ for $x\in X_i$ and $x\mapsto 1$ for  $x\in X-X_i$.  We have that   $\left\{\tau_{ij}\right\}_{i,j\in I}$ forms a projective system. Hence we get $F^{C^{u}}(X)\rightarrow\mathop{\varprojlim}\limits F^{C^{u}}(X_i) $. By \cite{Fa2008}Lemma 17.4.9(b), it is isomorphic. Therefore,  we get $\tau_i^{u}: \mathcal{F}^{u}\simeq \plim{i\in I}F^C(X_i)^u$. For each $i\in I$, consider the following commutative diagram for $X_i$.
\[
\xymatrix{
Z_{\mathcal{F}^m}(x^\alpha) \ar@{}[r]|-*[@]{\subset} & \mathcal{F}^m\ar[r]^{\tau^m_i}\ar[d]_{p} & F^{C}(X_i)^m\ar[d]^{p_i} & Z_{F^{C}(X_i)^m}((\tau_i^m(x))^\alpha) \ar@{}[l]|-*[@]{\supset} \ar@{|->}[d]  \\
\langle x\rangle \ar@{}[r]|-*[@]{\subset} & \mathcal{F}^{m-1}\ar[r]^{\tau^{m-1}_i} & F^{C}(X_i)^{m-1} & \langle \tau_i^{m-1}(x)\rangle \ar@{}[l]|-*[@]{\supset}
}
\]
 As $|X_i|< \infty$, we obtain $p_i(Z_{F^C(X_i)^m}((\tau_i^{m}(x) )^\alpha))\subset \langle \tau_i^{m-1}(x) \rangle$. Therefore, $\tau^{m-1}_i\circ p(Z_{\mathcal{F}^m}(x^\alpha))\subset \langle \tau_i^{m-1}(x) \rangle$ holds for all $i\in I$. Observe that $\plim{i\in I}{\langle \tau_i^{m-1}(x) \rangle}\subset \mathcal{F}^{m-1}$ is equal to $\langle x\rangle\subset \mathcal{F}^{m-1}$. So, we obtain $p(Z_{\mathcal{F}^m}(x^\alpha)) \subset\langle x\rangle$, hence $Z_{\mathcal{F}^m}(x^\alpha) \subset\langle x\rangle\cdot \mathcal{F}^{[m-1]}/\mathcal{F}^{[m]}$.
\end{proof}

In Proposition \ref{1.1.5}, if $\alpha\in \mathbb{Z}-\left\{0\right\}$, we have a more accurate and stronger result. (See Proposition \ref{1.1.9} below.) For a  pro-$C$ group $G$, we define the completed group ring of $G$ as $\mathbb{Z}_C[[G]]= \plim{H\overset{\op}\lhd G}\mathbb{Z}_C[G/H]$. The conjugate action $\mathcal{F}^{m-1}\curvearrowright\mathcal{F}^{[m-1]}/\mathcal{F}^{[m]}$ naturally extends to an action $\mathbb{Z}_C[[\mathcal{F}^{m-1}]]\curvearrowright\mathcal{F}^{[m-1]}/\mathcal{F}^{[m]}$, by which $\mathcal{F}^{[m-1]}/\mathcal{F}^{[m]}$ is regarded as a $\mathbb{Z}_C[[\mathcal{F}^{m-1}]]$ module.\par 

  Y. Ihara gave  Theorem A-1 and Theorem A-2 in \cite{Ih1999} about Blanchfield-Lyndon theory for free profinite groups. They are equivalent to the existence of  the following exact sequence.

\begin{equation}\label{appeq1}
1\rightarrow  \mathcal{M}^{\text{ab}}\rightarrow  \bigoplus_{1\leq i\leq r}\hat{\mathbb{Z}}[[\hat{F}/ \mathcal{M}]](x_{i}-1)\overset{\tilde{f}}\rightarrow \hat{\mathbb{Z}}[[\hat{F}/ \mathcal{M}]]\overset{\tilde{s}}\rightarrow \hat{\mathbb{Z}}\rightarrow 1 
\end{equation}
Here,  $\hat{F}$ is the profinite completion of  a free discrete group $F$ with basis $\left\{x_1,\cdots,x_r\right\}$,   and $ \mathcal{M}$ is a closed normal subgroup of $\hat{F}$, $\bigoplus_{1\leq i\leq r}\hat{\mathbb{Z}}[[\hat{F}/ \mathcal{M}]](x_{i}-1)$ is a free $\hat{\mathbb{Z}}[[\hat{F}/ \mathcal{M}]]$-module with basis $\left\{x_{i}-1\right\}_{1\leq i\leq r}$, $\tilde{f}$ is the sum of all components and  $\tilde{s}$ is the augmentation homomorphism.  The following proposition is to  give  Blanchfield-Lyndon theory for free pro-$C$ groups for an arbitrary non-trivial full class $C$ of finite groups. \par

\begin{prop}\label{bLtheory}
We set  $X=\left\{x_1,\cdots,x_r\right\}$ (in particular, assume that $|X|$ is finite). Let $\mathcal{N}$ be a closed normal subgroup of $\mathcal{F}$. Then the following exact sequence of  $\mathbb{Z}_C[[\mathcal{F}/\mathcal{N}]] $-modules exists.
\begin{equation}\label{appeq2}
1\rightarrow \mathcal{N}^{\text{ab}}\rightarrow \bigoplus_{1\leq i\leq r}\mathbb{Z}_{C}[[\mathcal{F}/\mathcal{N}]](x_{i}-1)\overset{f}\rightarrow \mathbb{Z}_C[[\mathcal{F}/\mathcal{N}]]\overset{s}\rightarrow \mathbb{Z}_{C}\rightarrow 1 
\end{equation}
Here, $ \bigoplus_{1\leq i\leq r}\mathbb{Z}_{C}[[\mathcal{F}/\mathcal{N}]](x_{i}-1)$ is a free $\mathbb{Z}_C[[\mathcal{F}/\mathcal{N}]] $-module with basis $\left\{x_{i}-1\right\}_{1\leq i\leq r}$, $f $ is the sum of all components and  $s$ is the augmentation homomorphism. 
\end{prop}
\begin{proof}
Let $\rho: \hat{F}\rightarrow \mathcal{F}$ be the natural surjection.  Since all terms of (\ref{appeq1}) is abelian, the following sequence induced by (\ref{appeq1}) is also exact.
\begin{equation}\label{appeq3}
1\rightarrow \left(\rho^{-1}(\mathcal{N})^{\text{ab}}\right)^{C}\rightarrow  \left(  \bigoplus_{1\leq i\leq r}\hat{\mathbb{Z}}[[\hat{F}/\rho^{-1}(\mathcal{N})]](x_{i}-1)\right)^{C}\xrightarrow{\tilde{f}^{C}} \left(\hat{\mathbb{Z}}[[\hat{F}/\rho^{-1}(\mathcal{N})]]\right)^C\xrightarrow{\tilde{s}^{C}} \mathbb{Z}_C\rightarrow 1 
\end{equation}
Clearly,  $(\hat{\mathbb{Z}}[[\hat{F}/\rho^{-1}(\mathcal{N})]])^{C}=\mathbb{Z}_{C}[[\hat{F}/\rho^{-1}(\mathcal{N})]]=\mathbb{Z}_{C}[[\mathcal{F}/\mathcal{N}]]$,  and  $\tilde{f}^{C}$ and $\tilde{s}^{C}$ coincide with $f$ and $s$ in (\ref{appeq2}), respectively. As is well known, abelianization commutes with taking the maximal pro-$C$ quotients. Therefore,   $(\rho^{-1}(\mathcal{N})^{\text{ab}})^C\cong (\rho^{-1}(\mathcal{N})^{C})^{\text{ab}}\cong \mathcal{N}^{\text{ab}}$.  $\hat{\mathbb{Z}}[[\hat{F}/\rho^{-1}(\mathcal{N})]] $ acts on all terms of (\ref{appeq3}) and these actions factor through  $\mathbb{Z}_C[[\mathcal{F}/\mathcal{N}]] $. Thus, (\ref{appeq2}) is an exact sequence of $\mathbb{Z}_C[[\mathcal{F}/\mathcal{N}]] $-modules. 
\end{proof}

\begin{lem}\label{1.1.7}
Let $x\in X$. Assume that $|X|$ is finite. Then the following hold.
\begin{enumerate}[(1)]
\item Let  $n\in\mathbb{N}$. Then  $\mathbb{Z}_{C}[[\mathcal{F}^{1}]]\ni x^n-1$ is a non-zero-divisor.
\item Let $S$ be a non-empty subset of  $\text{prime}(C)$.  Define     $\gamma=(\gamma_p)_{p\in \text{prime}(C)}\in \mathbb{Z}_C$ as  $\gamma_{p}=0$ (resp. $\gamma_{p}=1$) when $p\in S$ (resp. $p\notin S$).  Then  $\mathbb{Z}_{C}[[\mathcal{F}^{1}]]\ni x^\gamma-1\ $ is a zero-divisor.
\end{enumerate}
\end{lem}

\begin{proof}
Since  $\mathcal{F}$ is  a free pro-$C$ group with basis $X$, we have an isomorphism $\mathcal{F}^{1}\cong\underset{X}\bigoplus\mathbb{Z}_{C}$. Via this isomorphism, we identity $\mathcal{F}^{1}$ with $\underset{X}\bigoplus\mathbb{Z}_{C}$.\\
(1) $x$ is invertible in $\mathbb{Z}_{C}[[\mathcal{F}^{1}]]$ and $x^{-n}-1=-x^{-n}(x^{n}-1)$. Thus, we may assume that $n>0$.\par
Set $A:=\mathbb{Z}_{C}[[\underset{X-\left\{x\right\}}\bigoplus\mathbb{Z}_{C}]]$. We get $\mathbb{Z}_{C}[[\mathcal{F}^{1}]]=\plim{N\in\mathbb{N}(C)}A[\mathbb{Z}/N\mathbb{Z}]$ by definition. Let $y=(y_N)_{N}$ be an element of $\mathbb{Z}_{C}[[\mathcal{F}^{1}]]$ that satisfies  $(x^n-1)y=0$. We fix $N\in\mathbb{N}(C)$. Set $\rho:\mathbb{Z}_{C}[[\mathcal{F}^{1}]]\rightarrow A[\mathbb{Z}/N\mathbb{Z}]$ and $y_N={\displaystyle\sum^{N-1}_{i=0}} c_{i}^{N}\rho(x)^{i}\in A[\mathbb{Z}/N\mathbb{Z}]$.  By assumption, we have that  $0=(\rho(x)^n-1)y_N=(\rho(x)^n-1)(\sum c_i^N\rho(x)^i)$. Thus, $\sum(c_{i-n}^{N}-c_{i}^{N})\rho(x)^i=0$ and then $c_{i-n}^N=c_i^N$. (Here, we identify the set of subscripts $\left\{0,1,\cdots,i,\cdots,N-1\right\} $ with $\mathbb{Z}/N\mathbb{Z}$.) In other words, if $\overline{i}-\overline{i}' \in\langle n\rangle\subset  \mathbb{Z}/N\mathbb{Z}$, Then $c_{i}^{N}=c_{i'}^{N}$. \par
Set $n=n'n''(n'\in\mathbb{N}(C),n''$ coprime to all elements of $\text{prime}(C)$) and  take any $M\in \mathbb{N}(C)$ that satisfies $n'\mid M$ and any $k\in\mathbb{N}(C)$. As $\langle n\rangle=\langle n'\rangle \subset \mathbb{Z}/kM\mathbb{Z}$, we get $c_{i-n'}^{kM}=c_i^{kM}$, hence $c_{i-M}^{kM}=c_i^{kM}$ holds. Considering the projection $A[\mathbb{Z}/kM\mathbb{Z}]\rightarrow A[\mathbb{Z}/M\mathbb{Z}]$, we get $c_i^M=kc_i^{kM}$. (See the following diagram (\ref{diagram1.2}).)
\begin{eqnarray}\label{diagram1.2}
A[\mathbb{Z}/kM\mathbb{Z}]\ \ \ \ \ \ \ \ \ \ \ \ &\longrightarrow&\ \ \ \ \ \ \ \ \ \  A[\mathbb{Z}/M\mathbb{Z}]\nonumber \\
\begin{matrix}
c_0^{kM} 	&c_M^{kM} 	&\cdots 	& c_{(k-1)M}^{kM} \\
\vdots 	&\vdots&\ddots 	&\vdots 		\\
c_{M-1}^{kM} 	&c_{2M-1}^{kM} &\cdots 	& c_{kM-1}^{kM} \\
\end{matrix}
&\longmapsto&
\begin{matrix}
c_0^{M} =c_0^{kM}+\cdots+c_{(k-1)M}^{kM}=kc_0^{kM}\\
\vdots 	\\
c_{M-1}^M=c_{M-1}^{kM}+\cdots+c_{kM-1}^{kM} =kc_{M-1}^{kM}\\
\end{matrix}
\end{eqnarray}
Thus, $c_i^M\in \underset{k\in\mathbb{N}(C)}\bigcap k\cdot A=\left\{0\right\}$ for all $i$, hence $y_M=0$. Because $\left\{M\in\mathbb{N}(C)\mid n'|M\right\}$ is a cofinal subset of $\mathbb{N}(C)$, we get $y=0$.\\
(2) When $S= \text{prime}(C)$, the assertion is clear because $\gamma=0$. Hence we may assume that $\text{prime}(C)-S\neq \emptyset$. (In particular,  $|\text{prime}(C)|\geq 2$.)   Let $G:=\prod_{p\notin S}\mathbb{Z}_{p}$. We have that
\[
x^{\gamma}-1\in\mathbb{Z}_{C}[[G]]\subset \mathbb{Z}_{C}[[\mathbb{Z}_{C}]]]\subset \mathbb{Z}_{C}[[\mathcal{F}^{1}]].
\]
If an element of $\mathbb{Z}_{C}[[G]]$ is a zero-divisor, then it is also a zero-divisor of $\mathbb{Z}_{C}[[\mathcal{F}^{1}]]$. Thus, we may assume that $|X|=1$.  Let $\mathcal{N}:=\text{Ker}(\pi:\mathcal{F}\twoheadrightarrow G)$. By Theorem A.2, we get that $\mathcal{N}^{\text{ab}}\xrightarrow{\sim}\{a\in \mathbb{Z}_{C}[[G]]\mid a(\pi(x)-1)=0\}$. (We have that $x^{\gamma}=\pi(x)$ in $\mathbb{Z}_{C}[[G]]$.)  Since $\mathcal{N}^{\text{ab}}\twoheadrightarrow\prod_{p\in S}\mathbb{Z}_{p}\neq \{0\}$, $x^{\gamma}-1$ is zero-devisor in $\mathbb{Z}_{C}[[G]]$. This implies that  $x^{\gamma}-1$ is also  zero-devisor in $\mathbb{Z}_{C}[[\mathcal{F}^{1}]]$.
\end{proof}

\begin{lem}\label{1.1.6}
Let $x\in X$ and $\alpha\in\mathbb{Z}_C-\left\{0\right\}$. Assume that $|X|$ is finite and $|X|\geq 2$. Then the following hold.
\begin{equation*}
\mathbb{Z}_{C}[[\mathcal{F}^{1}]]\ni x^\alpha-1\ \text{is a non-zero-divisor }\iff Z_{\mathcal{F}^2}(x^\alpha)=\langle x\rangle
\end{equation*}
\end{lem}

\begin{proof}
Set $x_1:=x$ and $X:=\left\{x_1,\cdots,x_r\right\}$. We consider the following commutative diagram.
\begin{equation}\label{eq1.3}
\vcenter{
\xymatrix{
\mathbb{Z}_{C}[[\mathcal{F}^{1}]]\ \ar@{^{(}-_>}[r]^{\tau_1}\ar[d]^{\psi_1}&\mathcal{F}^{[1]}/\mathcal{F}^{[2]}\ar@{^{(}-_>}[r]^{\tau_2}\ar[d]^{\Psi_{x^{\alpha}-1}}&\mathbb{Z}_{C}[[\mathcal{F}^{1}]]^{\oplus r}\ar[d]^{\psi_2} \\
\mathbb{Z}_{C}[[\mathcal{F}^{1}]]\ar@{^{(}-_>}[r]^{\tau_1}&\mathcal{F}^{[1]}/\mathcal{F}^{[2]}\ar@{^{(}-_>}[r]^{\tau_2}&\mathbb{Z}_{C}[[\mathcal{F}^{1}]]^{\oplus r}
}}
\end{equation}
Here, all vertical arrows are multiplication by $x^{\alpha}-1$, the injection $\tau_2$ is induced by the  isomorphism of $\mathbb{Z}_{C}[[\mathcal{F}^1]]$-modules $\mathcal{F}^{[1]}/\mathcal{F}^{[2]}\cong \set{(a_1,...a_r)\in \mathbb{Z}_{C}[[\mathcal{F}^{1}]]^{\oplus r}}{\Sigma_{i=1}^r a_i(x_i-1)=0}$, the isomorphism is obtained in the case $\mathcal{N}=\mathcal{F}^{[1]}$ of Proposition \ref{bLtheory}, and the map $\tau_1$ is defined to send $\beta\in\mathbb{Z}_{C}[[\mathcal{F}^{1}]]$ to $(\beta(x_2-1),-\beta(x_1-1),0,\cdots,0)\in\mathcal{F}^{[1]}/\mathcal{F}^{[2]}$. (Observe that $\tau_1$ is injective by Lemma \ref{1.1.7}(1).) \par
By definition,  we have that ``$\mathbb{Z}_{C}[[\mathcal{F}^{1}]]\ni x^\alpha-1\text{ is a non-zero-divisor }\Leftrightarrow\text{Ker}(\psi_1)=\left\{0\right\} \Leftrightarrow\text{Ker}(\psi_2)=\left\{0\right\}$''. The diagram (\ref{eq1.3}) implies  ``$\text{Ker}(\psi_1)=\left\{0\right\} \Leftrightarrow\text{Ker}(\psi_2)=\left\{0\right\} \Leftrightarrow\text{Ker}(\Psi_{x^{\alpha}-1})=\left\{0\right\}''. $ By Proposition \ref{1.1.5} and the fact that $\langle x\rangle\subset Z_{\mathcal{F}^2}(x^\alpha)$, we get that  $Z_{\mathcal{F}^2}(x^\alpha)=\langle x\rangle\cdot(Z_{\mathcal{F}^2}(x^\alpha)\cap\mathcal{F}^{[1]}/\mathcal{F}^{[2]})$. Since we have that $Z_{\mathcal{F}^2}(x^\alpha)\cap\mathcal{F}^{[1]}/\mathcal{F}^{[2]}= \left\{ h\in \mathcal{F}^{[1]}/\mathcal{F}^{[2]}\mid x^{\alpha}hx^{-\alpha}=h \right\}=\text{Ker}(\Psi_{x^{\alpha}-1})$. Hence  we obtain that  $\text{Ker}(\Psi_{x^{\alpha}-1})=\left\{0\right\}\Leftrightarrow Z_{\mathcal{F}^2}(x^\alpha)=\langle x\rangle.$
\end{proof}

\begin{lem}\label{1.1.8}
Let $x\in X$. Then the following hold.
\begin{enumerate}[(1)]
\item Let  $n\in \mathbb{Z}-\left\{0\right\}$. Then $Z_{\mathcal{F}^2}(x^n) = \langle x\rangle$
\item Assume that $|X|\geq 2$. Let  $S \subset  \text{prime}(C)$ and $\gamma \in \mathbb{Z}_C$ as in Lemma \ref{1.1.7}.  Then  $Z_{\mathcal{F}^2}(x^\gamma) \varsupsetneqq \langle x\rangle$
\end{enumerate}
\end{lem}

\begin{proof}
When $|X|=1$, the assertion (1) is clear. When $|X|$ is finite, the assertions follow from Lemma \ref{1.1.7} and Lemma \ref{1.1.6}. The case $|X|=\infty$ is reduced to the case $|X|<\infty$, just similarly as at the end of the proof of Proposition \ref {1.1.5}.
\end{proof}

The next proposition is the main result of this subsection. 

\begin{prop}\label{1.1.9}
Assume that $m\geq 2$. Then $Z_{\mathcal{F}^m}(x^n) = \langle x\rangle$ holds for all $x\in X$ and all $n\in \mathbb{Z}-\left\{0\right\}$.
\end{prop}

\begin{proof}
The case $|X|=\infty$ is reduced to the case $|X|<\infty$, just similarly as at the end of the proof of Proposition \ref {1.1.5}. Hence we may assume  that $|X|<\infty$. When $|X|=1$, the assertion is clear. Hence we may assume that $|X|\neq 1$.\par
We prove the assertion by induction on $m \geq 2$. If $m=2$, the assertion holds for Lemma \ref{1.1.8}. Suppose that $m>2$ and that the assertion holds for $m-1$. To prove the assertion, it is sufficient to show that $Z_{\mathcal{F}^m}(x^n)\cap\mathcal{F}^{[m-1]}/\mathcal{F}^{[m]}$ is trivial  by Proposition \ref{1.1.5}, .\par
Let $\tilde{H}$ be a open subgroup of $\mathcal{F}$ that contain $\mathcal{F}^{[1]}$. Let $H$ be   the image of $H$ in $\mathcal{F}^{m}$. Let $N$ be the order of $x$ in $\mathcal{F}/\tilde{H}$ (=$\mathcal{F}^{m}/H$). By the Nielsen-Schreier theorem, $\tilde{H}$ is a free pro-$C$ group and  $x^{N}\in \tilde{H}$ is an element of a basis of $\tilde{H}$. Hence  $x^N\in H^{m-1}$ is also an element of a basis of $H^{m-1}$. (Note that $\tilde{H}^{m-1}\cong H^{m-1}$ because $\mathcal{F}^{[1]}/\mathcal{F}^{[m]}\subset H$.) By assumption of induction, we obtain  $Z_{H^{m-1}}((x^N)^n) = \langle x^N\rangle$ for all $n\in \mathbb{Z}-\left\{0\right\}$.  Set $\rho: H\twoheadrightarrow H^{m-1}$ and $\tilde{\rho}: \tilde{H}\twoheadrightarrow \tilde{H}^{m-1}$. 
 Since $m>2$, we have that $ \mathcal{F}^{[m-1]}\subset \tilde{H}^{[1]}$. Thus, $\tilde{\rho}(\mathcal{F}^{[m-1]})\subset \tilde{H}^{[1]}/\tilde{H}^{[m-1]}$ and hence $\langle x^N\rangle\cap \tilde{\rho}(\mathcal{F}^{[m-1]})=\left\{1\right\}$. We have that   $Z_{\mathcal{F}^m}(x^n)\subset Z_{\mathcal{F}^m}(x^{Nn})$ and $\rho(Z_{\mathcal{F}^m}(x^{Nn})\cap H)\subset Z_{H^{m-1}}(x^{Nn}) = \langle x^N\rangle$. Thus, we obtain that $\rho(Z_{\mathcal{F}^m}(x^n)\cap\mathcal{F}^{[m-1]}/\mathcal{F}^{[m]})\subset  \langle x^N\rangle\cap\rho(\mathcal{F}^{[m-1]}/\mathcal{F}^{[m]})=\left\{1\right\}.$ Considering all $H$, we get 
\[
 Z_{\mathcal{F}^m}(x^n)\cap\mathcal{F}^{[m-1]}/\mathcal{F}^{[m]} \subset \bigcap_{\mathcal{F}^{[1]}/\mathcal{F}^{[m]}\subset H\overset{\op}\subset \mathcal{F}^m } H^{[m-1]}=(\mathcal{F}^{[1]}/\mathcal{F}^{[m]})^{[m-1]}=\left\{1\right\}
.\]
Therefore, we obtain  $Z_{\mathcal{F}^m}(x^n) = \langle x\rangle$ by induction. 
\end{proof}
For a profinite  group $G$, we say that $G$ is slim if, for any open subgroup $H$ of $G$, $H$ is center-free. Clearly, if $G$ is slim, then $G$ is center-free. 

\begin{cor}
Assume that $m\geq 2$ and  $|X|\geq 2$. Then $\mathcal{F}^m$ is slim.
\end{cor}

\begin{proof}
Let $H$ be an open subgroup of $\mathcal{F}^{m}$ and $Z(H)$ the center of $H$.  Let $x,x'\in X$ be two distinct elements. There exist $n,n'\in\mathbb{N}$ such that $x^{n}\in H$ and $x'^{n'}\in H$. Since  $\langle x\rangle$ and $\langle x'\rangle$ are sent injectively by $\mathcal{F}^m\twoheadrightarrow \mathcal{F}^1$, we have that $ \langle x\rangle \cap \langle x'\rangle=\left\{1\right\}$.  Thus, we obtain that  $Z(H)\subset  Z_{H}(x^{n}) \cap Z_{H}(x'^{n'})\subset Z_{\mathcal{F}^m}(x^{n}) \cap Z_{\mathcal{F}^m}(x'^{n'})=\langle x\rangle\cap\langle x'\rangle =\left\{1\right\}$ by Proposition \ref{1.1.9}.
\end{proof}

Let $x$ be an element of $X$. We write   $N_{\mathcal{F}^m}(\langle x\rangle) :=\left\{f\in \mathcal{F}^m\mid f\langle x\rangle f^{-1}=\langle x\rangle\right\}$ and $\text{Comm}_{\mathcal{F}^m}(\langle x\rangle) :=\left\{f\in \mathcal{F}^m\mid f\langle x\rangle f^{-1}\text{ and }\langle x\rangle\text{ are commensurable}\right\}$.
\begin{cor}\label{1.1.10}
Assume that $m\geq 2$.  Then  \[
\langle x\rangle =Z_{\mathcal{F}^m}(x)=N_{\mathcal{F}^m}(\langle x\rangle)=\text{Comm}_{\mathcal{F}^m}(\langle x\rangle)\] 
hold for all $x\in X$.
\end{cor}

\begin{proof}
If $|X|=1$, then the assertion is clear. Assume that $|X|\geq 2$. Since $\langle x\rangle\subset Z_{\mathcal{F}^m}(x)\subset N_{\mathcal{F}^m}(\langle x\rangle)\subset \text{Comm}_{\mathcal{F}^m}(\langle x\rangle)$, we have only to show that $\langle x\rangle =\text{Comm}_{\mathcal{F}^m}(\langle x\rangle)$.  Let $f\in\mathcal{F}^{m}$ which satisfies that  $\langle x\rangle$ and $f\langle x\rangle f^{-1}$ are commensurable. Then there exists $n\in\mathbb{N}$ such that $ f\langle x\rangle f^{-1}\ni x^{n}$. Hence, we get $x^{n}=fx^{\alpha}f^{-1}$ for some $\alpha\in\mathbb{Z}_C$.
Note that $\langle x\rangle $ is mapped injectively by $\mathcal{F}^{m}\twoheadrightarrow \mathcal{F}^{1}$ and $x^{n}\equiv fx^{\alpha}f^{-1}\equiv x^{\alpha}\mod{\mathcal{F}^{[1]}/\mathcal{F}^{[m]}}$. Hence, $\alpha=n$. Therefore, we obtain that $f\in  Z_{\mathcal{F}^m}(x^{n})=\langle x\rangle$ by Proposition \ref{1.1.9}. 
\end{proof}

%%%%%%%%%%%%%%%%%%%%%%%%%%%%%%%%%%%%%%%%%%%%%%%%%%%%%%%%%%%%%%%%%%%%%%%%%%%%%%%%%%%%%%%%%%%%%%%

\subsection{Basic properties  of inertia groups and decomposition groups at cusps}\label{subsection1.2}
\hspace{\parindent}
In this subsection, we show the basic properties of inertia groups and decomposition groups at cusps of fundamental groups of curves.  First, we introduce some notations.
\begin{notation}\label{notation3}
From now on, we fix the following notations.
\begin{itemize}
\item For any scheme $S$, we write $K(S)$ for the function field of $S$.
\item $C^{\nil}$ is defined as the class of all nilpotent groups contained in $C$. For each $\ell\in$prime($C$), $C^{\pro\ell}$ is defined as the class of all $\ell$-groups contained in $C$. (Note that $C^{\pro\ell}$ coincides with the class of all $\ell$-groups, since $C$ is full and $\ell\in\text{prime}(C)$.)
\item Let $k$ be a field, $\overline{k}$ an algebraic closure of $k$ and $k^{\sep}$ the separable closure of $k$ in $\overline{k}$. We set $p:=\text{ch}(k)\geq0$ and assume $p\not\in \text{prime}(C)$. Let $X$ be a proper smooth curve over $k$  and $E$ a closed subgroup of $X$ which is finite \'{e}tale over $k$. Set $U:=X-E$.  For each field extension  $L/k$, we write  $U_{L}:=U\times_{k}L$.  Let $g:=g(U)$ be the  genus of $X_{\overline{k}}$, and $r:=r(U):=|E_{\overline{k}}|$.  For any curve $V$ over $k$, we write   $V^{\text{cpt}}$ for  the  regular compactification of $V$ (which is unique up to isomorphism). 

\item Fix an algebraically closed field $\Omega$ containing $K(U_{\overline{k}})$, which induces a geometric point  $\overline{\eta}:\text{Spec}(\Omega)\rightarrow U_{\overline{k}}$ over the generic point of $U_{\overline{k}}$.   Set $\mathcal{P}\in\left\{(\text{unrestricted}),\nil,\pro\ell\right\}$. Let $\mathcal{R}^{\mathcal{P}}\subset\Omega$ be the maximal pro-$C^{\mathcal{P}}$ Galois extension of  $K(U_{k^{\sep}})$ in $\Omega$ unramified on $U$ and 
$\mathcal{R}^{\mathcal{P},m}\subset\Omega$ the maximal $m$-step solvable pro-$C^{\mathcal{P}}$ Galois extension of $K(U_{k^{\sep}})$ in $\Omega$ unramified on $U$.  Let $\mathscr{S}^{\mathcal{P},m}(U)_x$ be the set of all places of $\mathcal{R}^{\mathcal{P},m}$ above $x\in E$ and $\mathscr{S}^{\mathcal{P},m}(U):=\cup_{x\in E}\mathscr{S}^{\mathcal{P},m}(U)_x$. Note that $\mathscr{S}^{\mathcal{P},m}(U_{k^{\sep}})=\mathscr{S}^{\mathcal{P},m}(U)$. If there is no risk of confusion, we write $\mathscr{S}^{\mathcal{P},m}_x:=\mathscr{S}^{\mathcal{P},m}(U)_x$ and $\mathscr{S}^{\mathcal{P},m}:=\mathscr{S}^{\mathcal{P},m}(U)$.  Let $n\in \mathbb{N}$ with $m\geq n$. Set $(\mathcal{P}, \mathcal{Q})$ is either  ((unrestricted),(unrestricted)), $(\nil,\nil)$, $(\pro\ell,\pro\ell)$, $((\text{unrestricted}),\nil)$, $((\text{unrestricted}), \pro\ell)$ or $(\nil,\pro\ell)$. We denote the natural surjection $\Pi^{(\mathcal{P},m)}(U)\rightarrow \Pi^{(\mathcal{Q}, n)}(U)$  by $\Psi^{\mathcal{P},m}_{\mathcal{Q},n}(U)$ and  the natural surjection $\mathscr{S}^{\mathcal{P},m}(U)\rightarrow \mathscr{S}^{\mathcal{Q}, n}(U)$  by $\psi^{\mathcal{P},m}_{\mathcal{Q},n}(U)$.  If there is no risk of confusion, we write $\Psi^{\mathcal{P},m}_{\mathcal{Q},n}:=\Psi^{\mathcal{P},m}_{\mathcal{Q},n}(U)$ and $\psi^{\mathcal{P},m}_{\mathcal{Q},n}:=\psi^{\mathcal{P},m}_{\mathcal{Q},n}(U)$.

\item We set \[
\overline{\Pi}^{\mathcal{P}}(U):=\pi_1(U_{k^{\sep}},\overline{\eta})^{C^{\mathcal{P}}}(=\pi_1(U_{\overline{k}},\overline{\eta})^{C^{\mathcal{P}}}),\ \ \ \ \ \ \Pi^{({\mathcal{P}},m)}(U):=\pi_1(U,\overline{\eta})/\Ker (\pi_1(U_{k^{\sep}},\overline{\eta})\twoheadrightarrow\overline{\Pi}^{\mathcal{P}}(U)^{m}),
\]
and let $p_{U/k}:\Pi^{(m)}(U)\twoheadrightarrow G_k:=\text{Gal}(k^{\sep}/k)$ be the natural projection. By definition, we have $\overline{\Pi}^{\mathcal{P}}(U)=\text{Gal}(\mathcal{R}^{\mathcal{P}}/K(U_{k^{\sep}}))$ and $\Pi^{(\mathcal{P},m)}(U)=\text{Gal}(\mathcal{R}^{\mathcal{P},m}/K(U))$. We write $\overline{\Pi}^{\mathcal{P},m}(U):=\overline{\Pi}^{\mathcal{P}}(U)^{m}$. If there is no risk of confusion, we also write $\overline{\Pi}^{\mathcal{P},m}:=\overline{\Pi}^{\mathcal{P},m}(U),\ \overline{\Pi}^{\mathcal{P}}:=\overline{\Pi}^{\mathcal{P}}(U)$ and $\Pi^{(\mathcal{P},m)}:=\Pi^{(\mathcal{P},m)}(U)$.  Let $I_{y,\overline{\Pi}^{\mathcal{P},m}(U)}\ (\text{resp. }D_{y,\Pi^{(\mathcal{P},m)}(U)})$ be the stabilizer of $y\in \mathscr{S}^{\mathcal{P},m}(U)$ in $\overline{\Pi}^{\mathcal{P},m}(U)$ (resp. $\Pi^{(\mathcal{P},m)}(U)$) with respect to the natural action $\overline{\Pi}^{\mathcal{P},m}(U)\curvearrowright\mathscr{S}^{\mathcal{P},m}(U)\ (\text{resp. }\Pi^{(\mathcal{P},m)}(U)\curvearrowright\mathscr{S}^{\mathcal{P},m}(U))$. We call it the  inertia group (resp. the decomposition group) at $y$. 
 We define the following subsets of  $\overline{\Pi}^{\mathcal{P},m}(U)$.
\[
\mathcal{I}_{x,\overline{\Pi}^{\mathcal{P},m}(U)}:=\underset{y\in \mathscr{S}^{\mathcal{P},m}(U)_x} \cup I_{y,,\overline{\Pi}^{\mathcal{P},m}(U)}\ \ \left(x\in E_{k^{\text{sep}}}\right),\ \  \ \ \ \ \ \mathcal{I}_{\overline{\Pi}^{\mathcal{P},m}(U)}:=\underset{x\in E_{k^{\text{sep}}}}\cup\mathcal{I}_{x,\overline{\Pi}^{\mathcal{P},m}(U)}
\]
 If there is no risk of confusion, we write $I_y:=I_{y,\overline{\Pi}^{\mathcal{P},m}(U)}$, $D_y:=D_{y,\Pi^{(\mathcal{P},m)}(U)}$ and  $\mathcal{I}_{x}:=\mathcal{I}_{x,\overline{\Pi}^{\mathcal{P},m}(U)}$. From the assumption $p\not\in\text{prime}(C)$, we have
\begin{equation}\label{equation2.1}
\overline{\Pi}(U)\cong \text{the pro-}C\text{ completion of  }\left\langle \alpha_1,\cdots,\alpha_g,\beta_1,\cdots,\beta_g,\sigma_1,\cdots,\sigma_r\middle|\prod_{i=1}^{g}[\alpha_i,\beta_i]\prod_ {j=1}^{r}\sigma_j=1\right\rangle .
\end{equation}
Here, $\sigma_1,\ldots,\sigma_r$ are generators of  inertia groups. If $m\geq 2$, we have
\begin{eqnarray*}
\overline{\Pi}^{m}(U)\text{ is not abelian}&\iff&\overline{\Pi}^{\nil,m}(U)\text{ is not abelian} \\
&\iff&2-2g-r<0\iff(g,r)\notin\left\{(0,0),(0,1),(1,0),(0,2)\right\}
\end{eqnarray*}
\end{itemize}
\end{notation}

\begin{lem}\label{inersep1}
Assume that $r\geq 3$.  Let $y,y' \in \mathscr{S}^1(U)$. Then the following conditions are equivalent: (a) $\psi^{1}_{0}(y)=\psi^{1}_{0}(y')$, (b) $I_{y, \overline{\Pi}^{1}}$ and $I_{y', \overline{\Pi}^{1}}$ are commensurable, and (c) $I_{y, \overline{\Pi}^{1}}\cap I_{y', \overline{\Pi}^{1}}\neq\{1\}$.
\end{lem}
\begin{proof}
The implications (a)$\Rightarrow$(b)$\Rightarrow$(c) is clear. We show the implication (c)$\Rightarrow$(a). Assume that $\psi^{1}_{0}(y)\neq\psi^{1}_{0}(y')$.   By (\ref{equation2.1}), we obtain that $\overline{\Pi}^{\ab}\cong \mathbb{Z}_C\alpha_1\oplus\cdots\oplus\mathbb{Z}_C\alpha_g\oplus\mathbb{Z}_C\beta_1\oplus\cdots \oplus\mathbb{Z}_C\beta_g\oplus\mathbb{Z}_C\sigma_1\cdots \oplus\mathbb{Z}_C\sigma_{r-1}$. By the assumption $r\geq3$, we may assume that $I_y$ and $I_{y'}$ are mapped isomorphically onto $\mathbb{Z}_{C}\sigma_1$ and  $\mathbb{Z}_{C}\sigma_2$, respectively, by $\overline{\Pi}^{m}\twoheadrightarrow\overline{\Pi}^{\ab}$. Since the intersection of their images is trivial, we get $I_y\cap I_{y'}=\left\{1\right\}$.
\end{proof}
 If $r\geq 1$, then $\overline{\Pi}$ is a free pro-$C$ group by (\ref{equation2.1}). If  $r\geq 2$,  any generator of  an inertia group of $\overline{\Pi}$ is  an element of some basis of $\overline{\Pi}$. Thus, we may apply the results of subsection \ref{subsection1.1} to $\overline{\Pi}$ and its inertia groups  if  $r\geq 2$.  In the next lemma, we use this fact and Corollary \ref{1.1.10}.\par

\begin{lem}\label{inersepm}
Assume that $m\geq 2$, that $r\geq 2$,  and that $(g,r)\neq (0,2)$.    Let $y,y' \in \mathscr{S}^m(U)$. Consider the following conditions: (a)  $y=y'$,  (b) $I_{y, \overline{\Pi}^{m}}$ and $I_{y', \overline{\Pi}^{m}}$ are commensurable, and (c) $I_{y, \overline{\Pi}^{m}}\cap I_{y', \overline{\Pi}^{m}}\neq \{1\}$. Then one has  (a)$\Leftrightarrow$(b)$\Rightarrow$(c). If, moreover, $|\Sigma|=1$, then one has (a)$\Leftrightarrow$(b)$\Leftrightarrow$(c).
\end{lem}
\begin{proof}
The implications (a)$\Rightarrow$(b)$\Rightarrow$(c) is clear. First, we show the implication (b)$\Rightarrow$(a). By Lemma \ref{inersep1} (b)$\Rightarrow$(a), we may assume that  $y$ and $y'$ are above the same point of $E_{k^{\text{sep}}}$. Hence, there exists $f\in\overline{\Pi}^{m}$ such that $y'=fy$, in particular, $I_{y'}= fI_{y}f^{-1}$. (The action of $\overline{\Pi}^{m}$ on  $\mathscr{S}^m(U)_x$  is transitive for all $x\in E_{k^{\text{sep}}}$.) By  Corollary \ref{1.1.10}, we have $f\in I_y$. Thus,  $y=y'$  follows. Next, we show  (c)$\Rightarrow$(b) when  $|\Sigma|=\left\{\ell\right\}$. Since a closed subgroup of  $I_y(\cong\mathbb{Z}_{C}=\mathbb{Z}_{\ell})$  is trivial or open, $I_{y}\cap I_{y'}$ is trivial or  open in $I_y$. Thus,  (c)$\Rightarrow$(b) follows. \\
\end{proof}

\begin{rem}\label{1.2.4}
When   $m= 2$, $r\geq 2$,  $(g,r)\neq (0,2)$, and $|\text{prime}(C)|\geq2$, we can construct an example of $y,y'\in \mathscr{S}^2(U)$ such that $y\neq y'$ but $I_y\cap I_{y'}\neq \left\{1\right\}$. Indeed, take any $y\in E_{k^{\text{sep}}}$ and let $\sigma$ be a generator of  $I_{y}$.  By Lemma \ref{1.1.8}(2), there exists $f\in Z_{\overline{\Pi}^2}(\sigma^{\gamma})-I_y$. Set $y'\overset{\text{def}}=fy$, then $y\neq y'$ and $1\neq \sigma^{\gamma}\in I_{y}\cap fI_{y}f^{-1}=I_y\cap I_{y'}$.
\end{rem}

\begin{prop}\label{1.2.3}
Assume that  $m\geq 2$, $r\geq 2$,  and that  $(g,r)\neq (0,2)$. For  $ y\in \mathscr{S}^m(U)$,   $I_{y,\overline{\Pi}^{m}}=N_{\overline{\Pi}^m}(I_{y,\overline{\Pi}^{m}})$ and $D_{y,\Pi^{(m)}}=N_{\Pi^{(m)}}(I_{y,\overline{\Pi}^m})$ hold.
\end{prop}

\begin{proof}
 $I_{y,\overline{\Pi}^{m}}=N_{\overline{\Pi}^m}(I_{y,\overline{\Pi}^{m}})$ follows from Corollary \ref{1.1.10}. Hence, we have only to show  $D_{y,\Pi^{(m)}}=N_{\Pi^{(m)}}(I_{y,\overline{\Pi}^m})$.  Since $I_y\lhd D_y$, $D_{y}\subset N_{\Pi^{(m)}}(I_{y})$ holds. If  $\tau\in N_{\Pi^{(m)}}(I_{y})$, then $I_y=\tau I_y\tau^{-1}=I_{\tau y}$. So we obtain  $\tau y=y$ by  Lemma \ref{inersepm}(b)$\Rightarrow$(a), hence $\tau\in D_{y}$. Thus, $D_{y}=N_{\Pi^{(m)}}(I_{y})$.
\end{proof}

The following variant of Proposition \ref{1.2.3} with $\Pi^{(m)}$ replaced by $\Pi^{(\nil,m)}$ is not contained in Lemma \ref{inersepm} and  Proposition \ref{1.2.3} , since $C^{\nil}$ is not full.

\begin{prop}\label{1.2.7}
Assume that  $m\geq 2$, $r\geq 2$,  and that  $(g,r)\neq (0,2)$.  Then, for each pair $ y, y'\in \mathscr{S}^{\nil,m}(U)$ with $y\neq y'$, the following hold.
\begin{enumerate}[(1)]
\item  $\mathscr{S}^{\nil,m}$ is identified with  the fiber product over $E_{k^{\text{sep}}}$ of $\mathscr{S}^{\pro\ell,m}$ for all $\ell\in \text{prime}(C)$.
\item $I_{y,\overline{\Pi}^{\nil,m}}$ and $I_{y',\overline{\Pi}^{\nil,m}}$ are not commensurable.
\item$I_{y,\overline{\Pi}^{\nil,m}}=N_{\overline{\Pi}^{\nil,m}}(I_{y,\overline{\Pi}^{\nil,m}})$ and $D_{y,\Pi^{(\nil,m)}}=N_{\Pi^{(\nil,m)}}(I_{y,\overline{\Pi}^{\nil,m}})$.
\end{enumerate}
\end{prop}
\begin{proof}
(1) The assertion is equivalent to $\mathscr{S}^{\nil,m}_{x}=\prod_{\ell}\mathscr{S}^{\pro\ell,m}_{x}$ for all $x\in E_{k^{\text{sep}}}$. Let $x\in E_{k^{\text{sep}}}$ and $w\in\mathscr{S}^{\nil,m}_{x}$. Then we have 
\[
\mathscr{S}^{\nil,m}_x=\overline{\Pi}^{\nil,m}/I_{w} =\left(\prod_{\ell}\overline{\Pi}^{\pro\ell,m}\right)/\left(\prod_{\ell}\Psi^{\nil,m}_{\pro\ell,m}(I_{w})\right)= \prod_{\ell}\left(\overline{\Pi}^{\pro\ell,m}/\Psi^{\nil,m}_{\pro\ell,m}(I_{w})\right)=
\prod_{\ell}\mathscr{S}^{\pro\ell,m}_{x}.
\]
Hence the assertion follows.\\
(2) Because  $\overline{\Pi}^{\text{nil},m}$  is equal to the product of $\overline{\Pi}^{\text{pro-}\ell,m}$  for all $\ell\in \text{prime}(C)$, we have only to show that  there exists  $\ell\in$prime($C$) such that  $I_{\psi^{\nil,m}_{\pro\ell,m}(y)}$ and $ I_{\psi^{\nil,m}_{\pro\ell,m}(y')}$ are not commensurable. By Proposition \ref{inersepm},  this condition is equivalent to saying that there exists $\ell\in$prime($C$) such that  
$\psi^{\nil,m}_{\pro\ell,m}(y)\neq \psi^{\nil,m}_{\pro\ell,m}(y')$. Thus, the assertion follows from  $y\neq y'$ and (1).\\
(3) First, we show the first assertion. Because  $\overline{\Pi}^{\text{nil},m}$  is equal to the product of $\overline{\Pi}^{\text{pro-}\ell,m}$  for all $\ell\in\text{prime}(C)$,  we have 
\begin{equation*}
N_{\overline{\Pi}^{\nil,m}}(I_{y})\xrightarrow{\sim} \prod_{\ell\in\text{prime}(C)}N_{\overline{\Pi}^{\pro\ell,m}} (I_{\psi^{\nil,m}_{\pro\ell,m}(y)})=\prod_{\ell\in\text{prime}(C)} I_{\psi^{\nil,m}_{\pro\ell,m}(y)}\xleftarrow{\sim} I_{y}.
\end{equation*}
Here, the middle equality follows from Proposition \ref{1.2.3}. Next, we show the second assertion. Since $I_y\lhd D_y$,  $D_{y}\subset N_{\Pi^{(\nil,m)}}(I_{y})$ holds. If  $\tau\in N_{\Pi^{(\nil,m)}}(I_{y})$, then $I_y=\tau I_y\tau^{-1}=I_{\tau y}$. So we obtain  $\tau y=y$ by  (2), hence $\tau\in D_{y}$. Thus, $D_{y}=N_{\Pi^{(\nil,m)}}(I_{y})$  and the second assertion follows.\\
\end{proof}

%%%%%%%%%%%%%%%%%%%%%%%%%%%%%%%%%%%%%%%%%%%%%%%%%%%%%%%%%%%%%%%%%%%%%%%%%%%%%%%%%%%%%%%%%%%%%%%

\subsection{Weight filtration}\label{subsection1.3}
\hspace{\parindent}In this subsection, we define the ($\tau$-)weights of $\ell$-adic Galois representations. Let $\ell$ be a prime different from  $p$ and fix an isomorphism $\tau:\overline{\mathbb{Q}}_{\ell}\cong \mathbb{C}$. \par
Let  $w\in \mathbb{Z}$, a finite dimensional $\mathbb{Q}_{\ell}$-vector space $V$ and a continuous homomorphism $\phi:G_k\rightarrow GL(V)$. 
When $k$ is finite and    $Fr:\overline{k}\rightarrow \overline{k}$ is a $q(:=|k|)$-power Frobenius morphism, if any root $\alpha$ of det($Id-\phi (Fr)\cdot t$) satisfies $ |\tau (\alpha )|=q^{w/2} $,  then we say that $\phi$ has weight $w$. When $k$ is  an infinite  field finitely generated over the prime field,  if there exists a normal scheme $X$ of finite type over  $\text{Spec}(\mathbb{Z})$ that  satisfies the following conditions, we say that $\phi$ has weight $w$. (If $V=\left\{0\right\}$, we define  $\phi$ has weight $w$ for all $w$.)
\begin{enumerate}[(i)]
\item $k=K(X$).
\item   $G_{k}\to GL(V)$ factor through the natural morphism  $G_{k}\to \pi_1(X)$. 
\item The composite of $G_{\kappa(x)}\rightarrow \pi_1(X)\rightarrow GL(V)$ has weight $w$ for any closed point $x$ of $X$.
\end{enumerate}
Let $\tilde{V}$ be a torsion-free finitely generated $\mathbb{Z}_{\ell}$-module and $\tilde{\phi}:G_k\rightarrow GL(\tilde{V})$ a continuous homomorphism. If  the $G_k$-action  on $\tilde{V}\otimes\mathbb{Q}_{\ell}$ induced by $\tilde{\phi}$ has weight $w$, then we say that $\tilde{\phi}$ has weight $w$.  We write $W_{w}(\tilde{V})$ for the maximal $\mathbb{Z}_{\ell}$-submodule of $\tilde{V}$ that has weight $w$. (See \cite{Na1990-405}Proposition 2.1.)

\begin{prop}\label{1.3.5}Let $J(X_{k^{\text{sep}}})$ be the  Jacobian variety of $X_{k^{\text{sep}}}$. Then the following exact sequence of  $G_k$-modules exists.
\begin{equation}\label{wf}
0\rightarrow \mathbb{Z}_{\ell}(1)\rightarrow \mathbb{Z}[E_{k^{\text{sep}}}]\bigotimes_{\mathbb{Z}} \mathbb{Z}_{\ell}(1)\xrightarrow{\rho} \pi_1(U_{\overline{k}})^{\ab ,\pro\ell}\rightarrow T_{\ell}(J(X_{k^{\text{sep}}}))\rightarrow 0
\end{equation}
Here, $ \mathbb{Z}[E_{k^{\text{sep}}}]$ is a free $\mathbb{Z}$-module generated by $E_{k^{\text{sep}}}$ and regard it as a  $G_k$-module via the  $G_k$-action on   $E_{k^{\text{sep}}}.$
\end{prop}
\begin{proof}
\cite{Ta1997}Remark(1.3).
\end{proof}

By the following  proposition, we can compute the weights of all terms of (\ref{wf}). 

\begin{prop}\label{1.3.6}
Assume that   $k$ is finitely generated over the prime field. Then the $G_k$-action on $T_{\ell}(J(X_{k^{\text{sep}}}))$ has weight $-1$. In particular, if $\ell\in\text{prime}(C)$, then the image of the morphism $\rho$ in (\ref{wf})  coincides with  $W_{-2}(\overline{\Pi}^{\ab ,\pro\ell})$.
\end{prop}

\begin{proof}
The second assertion follows from the first assertion and  the fact that $\mathbb{Z}_{\ell}(1)$ has weight $-2$. When $k$ is a number field, the first assertion is shown in  \cite{Na1990-411}(2.7). When $k$ is a field  finitely generated over the prime field, the first assertion follows from the similar proof of the case that $k$ is a number field.
\end{proof}

%%%%%%%%%%%%%%%%%%%%%%%%%%%%%%%%%%%%%%%%%%%%%%%%%%%%%%%%%%%%%%%%%%%%%%%%%%%%%%%%%%%%%%%%%%%%%%%

\subsection{Reconstruction of inertia groups and decomposition groups at cusps}\label{subsection1.4}

\hspace{\parindent}In this section, we define the maximal cyclic subgroups of cyclotomic type and  we show that the inertia groups can be characterized as their images. \par
First, we give  the following lemmas which play an essential role in the reconstruction of inertia groups. \par

\begin{lem}\label{1.2.6}
Let $G$ be a profinite group, $z\in G-\left\{1\right\}$ and $\mathcal{J}$ a closed subset of $G$. Assume that $z\in [H,H]H^{\ell\text{-}th}\langle\mathcal{J}\cap H\rangle \text{ holds for all }H\overset{\op}\leq G\text{ with }z\in H \text{ and all prime }\ell $. Then we have $\langle z\rangle\cap \mathcal{J}\neq\left\{1\right\}$. Here, $H^{\ell\text{-}th}$ stands for the set of  the $\ell$-th powers of  all elements of $H$.
\end{lem}

\begin{proof}(See \cite{Na1990-411}Lemma 3.1.)
Let us show the contraposition. Assuming $ \langle z\rangle\cap \mathcal{J}=\left\{1\right\}$, we will construct $H$ and $\ell$.\par
Let $\ell$ be a prime satisfying $\langle z^{\ell}\rangle\subsetneq\langle z\rangle$. There exists $B\overset{\op}\leq G$ with $B\cap\langle z\rangle=\langle z^{\ell}\rangle$ by \cite{Na1994}Proposition 1.4.1(i). Note that  $\mathcal{J}-B$ and $\langle z\rangle-B$ are closed subsets of $G$ whose intersection is empty by assumption. First, we claim:  
\begin{equation}\label{式2}
\text{There exists }M\overset{\op}\lhd G \text{ such that }M\subset B\text{ and }(\mathcal{J}-B)\cap \ (M\cdot (\langle z\rangle-B))=\emptyset .
\end{equation}
Indeed,  for each $w\in\mathcal{J}-B$, there exists $W_{w}\overset{\op}\lhd G$ which satisfies  $w\notin W_{w}(\langle z\rangle-B)$. Since $\underset{w}\cap((\mathcal{J}-B)\cap W_w(\langle z\rangle-B))=\emptyset$ and $\mathcal{J}-B$ is compact, There exist $w_1,\cdots ,w_n$ such that $(\mathcal{J}-B)\cap( \underset{i}\cap W_{w_i}(\langle z\rangle-B)))=\emptyset$. Write $B_{G}$ for the normal core of $B$ (i.e. the intersection of  all conjugates of $B$ in $G$). Then $M:=\underset{i}\cap W_{w_i}\cap B_{G}$ satisfies the desired property.\par
Set $H:=M\langle z\rangle$. Finally, show that  $H$ and $\ell$ satisfy the desired properties.
 Since $M\cap\langle z\rangle\subset B\cap\langle z\rangle= \langle z^{\ell}\rangle$, we obtain $M\cap \langle z\rangle =M\cap \langle z^{\ell}\rangle$. Thus,
\[
\xymatrix@R=10pt{
1\ar[r] &M\ar[r]&M\langle z\rangle\ar[r]&\langle z\rangle/(M\cap\langle z\rangle\ar[r])&1\\
1\ar[r]& M\ar[r]\aru{=}&M\langle z^{\ell}\rangle\ar[r]\aru{\subset }&\langle z^{\ell}\rangle/(M\cap\langle z^{\ell}\rangle)\ar[r]\ar@{^{(}->}[u] &1
}
\]
Hence$M\langle z^{\ell}\rangle \lhd M\langle z\rangle$ and $M\langle z\rangle/M\langle z^{\ell}\rangle\cong\mathbb{Z}/\ell\mathbb{Z}$. In particular, we have $[H:H]H^{\ell\text{-}th}\subset M\langle z^{\ell}\rangle$.
As $M\subset B$, we have  $(M(\langle z\rangle-B))=M\langle z\rangle-B$, hence, by (\ref{式2}), $\mathcal{J}\cap M\langle z\rangle\subset B$. Thus,
\[
\mathcal{J}\cap H=\mathcal{J}\cap M\langle z\rangle\subset B\cap M\langle z\rangle=M(B\cap\langle z\rangle)= M\langle z^{\ell}\rangle
\]
From these, we get $[H,H]H^{\ell\text{-}th}\langle\mathcal{J}\cap H\rangle\subset M\langle z^{\ell}\rangle\subset B$. Hence, we obtain $z\notin [H,H]H^{\ell\text{-}th}\langle\mathcal{J}\cap H\rangle$  as $z\notin B$.
\end{proof}

\begin{lem}\label{1.2.5}
Assume that $r\geq 2$. Let $z\in \overline{\Pi}^m$ satisfying $\langle z\rangle\cap \mathcal{I}_{\overline{\Pi}^m}\neq \left\{1\right\}$. Then  $z\in \mathcal{I}_{\overline{\Pi}^m}\cdot \overline{\Pi}^{[m-1]}/\overline{\Pi}^{[m]} $ holds.
If, moreover $|\text{prime}(C)|=1$, then $z\in \mathcal{I}_{\overline{\Pi}^m}$ holds.
\end{lem}

\begin{proof}
First, we show the first assertion. By assumption, there exists $\alpha\in\mathbb{Z}_C$ and  $y\in\mathscr{S}^m(U)$ such that $z^{\alpha}\in I_y-\left\{1\right\}$. Since  $z^{\alpha}=zz^{\alpha}z^{-1}$ and $r\geq2$, Proposition \ref{1.1.5} implies that  $ z\in Z_{\overline{\Pi}^m}(z^{\alpha})\subset  I_{y}\cdot \overline{\Pi}^{[m-1]}/\overline{\Pi}^{[m]}$.\par
Next, we show the second assertion. Set $\text{prime}(C)=\left\{\ell\right\}$. (Note that $\mathbb{Z}_{C}=\mathbb{Z}_{\ell}$.)
If  $(g,r)=(0,2)$, the assertion clearly holds because $\mathcal{I}_{\overline{\Pi}^m}=\overline{\Pi}^m$. So we may assume $(g,r)\neq(0,2)$.\par
Case  $m\geq 2$. (Note that  $\overline{\Pi}^{m}$ is not abelian.)  By assumption,  there exists $\alpha\in\mathbb{Z}_C$ and  $y\in\mathscr{S}^m(U)$ such that $z^{\alpha}\in I_y-\left\{1\right\}$.  Hence  $z^{\alpha}=zz^{\alpha}z^{-1}\in I_y\cap zI_yz^{-1}=I_y\cap I_{zy}$.  By Proposition \ref{inersepm}, we get $y=zy$. Thus,  $z\in I_{y}$ by definition. \par
 Case $m=1$.   There exists $\alpha \in \mathbb{Z}_{\ell}$ such that  $z^{\alpha}\in \mathcal{I}_{\overline{\Pi}^1}-\left\{1\right\}$ by the assumption. By (\ref{equation2.1}), $\overline{\Pi}^1$is a free $\mathbb{Z}_{\ell}$-module generated by $\alpha_1,\cdots,\alpha_g,\beta_1,\cdots,\beta_g,\sigma_1,\cdots,\sigma_{r-1}$  and  $\mathcal{I}_{\overline{\Pi}^1}= \langle\sigma_{1}\rangle \cup\cdots \cup\langle\sigma_{r}\rangle\subset \langle \sigma_1,\cdots, \sigma_{r-1}\rangle$. Then there exists $a_1,\cdots,a_g,b_1,\cdots,b_g,c_1,\cdots,c_{r-1}\in\mathbb{Z}_{\ell}$ such that $z=\Sigma a_i\alpha_i+\Sigma b_i\beta_i+\Sigma c_j\sigma_j$. As $z^{\alpha}\in\mathcal{I}_{\overline{\Pi}}^1=\langle\sigma_{1}\rangle \cup\cdots \cup\langle\sigma_{r}\rangle$,  this implies either $z^{\alpha}\in \langle\sigma_{h}\rangle$ for some $1\leq h\leq r-1$ or $z^{\alpha}\in \langle \sigma_{r}\rangle=\langle\sigma_{1}+\cdots+\sigma_{r-1}\rangle$. Since $\mathbb{Z}_{\ell}$ is an integral domain, we deduce  $a_i=b_{i}=0$ for any $i$, and either $c_{1}=\cdots=c_{h-1}=c_{h+1}=\cdots=c_{r-1}=0$ or  $c_{1}=\cdots=c_{r-1}$.   Hence $z\in \langle\sigma_{1}\rangle \cup\cdots \cup\langle\sigma_{r}\rangle=\mathcal{I}_{\overline{\Pi}^1}$.
\par
\end{proof}

Maximal cyclic subgroups of cyclotomic type are first defined in  \cite{Na1990-411}Definition 3.3 in the case of the full fundamental group. Our definition differs from that of \cite{Na1990-411} for the following two points: 
\begin{inparaenum}[(i)]
\item We weaken the self-normalizing property in \cite{Na1990-411}; and
\item We generalize the definition from number fields to fields finitely generated over  the prime field of arbitrary characteristic. 
\end{inparaenum}
\begin{definition}\label{def1.4.1}
Let $k$ be a field finitely generated over the prime field and $J$ a closed subgroup of $\overline{\Pi}^m$. If $J$ satisfies the following conditions, then $J$ is called a  maximal cyclic subgroup of cyclotomic type.
\begin{enumerate}[(i)]
\item $J \cong \mathbb{Z}_C$
\item Write  $\overline{J}$ for the image   of  $J$ by $\overline{\Pi}^m\rightarrow \overline{\Pi}^{\ab}$. Then $J\overset{\sim}\rightarrow \overline{J}$ and  $\overline{\Pi}^{\ab}/\overline{J}$ is torsion-free.
\item $p_{U/k}(N_{\Pi^{(m)}}(J))\overset{\op}\leq G_k$
\item Let $\chi_{\text{cycl}}:G_k\rightarrow \mathbb{Z}_C^{\times}$ be the cyclotomic character and $N_{\Pi^{(m)}}(J)\to \text{Aut}(J)=\mathbb{Z}_C^{\times}$ the character obtained from the conjugate action. Then the  following diagram is commutative.
\[
\xymatrix@R=10pt{
N_{\Pi^{(m)}}(J) \ar[r] \ar[d]_{p_{U/k}}& \text{Aut}(J)\\
G_k\ar[r]^{\chi_{cycl}}&\mathbb{Z}_C^{\times}\aru{=}\\
}
\]

\end{enumerate}
\end{definition}

Next, we give a group-theoretical characterization of the inertia groups for three cases (Proposition \ref{1.4.2}, Proposition \ref{1.4.3}, Proposition \ref{1.4.4}).\par

\begin{prop}\label{1.4.2}
Assume that   $r\geq 2$ and $k$ is  finitely generated over the prime field. Then for any subgroup $I$ of     
$\overline{\Pi}^m$, the following conditions are equivalent.
\begin{enumerate}[(a)]
\item $I$ is an inertia group.
\item $I$ is the image of a maximal cyclic subgroup of cyclotomic type of $\overline{\Pi}^{m+2}$  by the map $\overline{\Pi}^{m+2}\rightarrow \overline{\Pi}^{m}$.
\end{enumerate}
\end{prop}
\begin{proof}
First, we show (b)$\Rightarrow $(a). Let $J$ be  a maximal cyclic subgroup of cyclotomic type of $\overline{\Pi}^{m+2}$ whose image by the map $\overline{\Pi}^{m+2}\rightarrow \overline{\Pi}^{m}$ coincides with $I$.  Let $z$ be a generator of $J$ (cf. Definition \ref{def1.4.1}(i)), $z_{1}$ the image of  $z$ by $\overline{\Pi}^{m+2}\twoheadrightarrow\overline{\Pi}^{m+1}$, $\ell\in\text{prime}(C)$ and $H_{1}$ an open subgroup of $\overline{\Pi}^{m+1}$ that contains $z_{1}$.\par
Let $H\subset \overline{\Pi}^{m+2}$ and $\tilde{H}\subset \overline{\Pi}$  be the inverse images of $H_{1}$ by $\overline{\Pi}^{m+2}\twoheadrightarrow\overline{\Pi}^{m+1}$ and $\overline{\Pi}\twoheadrightarrow \overline{\Pi}^{m+1}$, respectively.
Since  we have $\ell\in \text{prime}(C$) and $\overline{\Pi}^{[m+1]}/\overline{\Pi}^{[m+2]}\subset H$ by definition, we have  $\tilde{H}^{\ab ,\pro\ell}=H^{\ab ,\pro\ell}$.   We get $\langle\mathcal{I}_{\tilde{H}^{\ab,\pro\ell}}\rangle=W_{-2}(\tilde{H}^{\ab ,\pro\ell})$ by Proposition \ref{1.3.6}, hence  $\langle\mathcal{I}_{H^{\ab ,\pro\ell}}\rangle = W_{-2}(H^{\ab ,\pro\ell})$.\par
Let $\overline{z}$ be the image of $z$ by $H\twoheadrightarrow H^{\ab ,\pro\ell}$.  Since the action of $p_{U/k}(N_{\Pi^{(m+2)}}(J))$ on $J$ is cyclotomic  by Definition \ref{def1.4.1}(iii)(iv), the action has weight $-2$.  Thus,  $\overline{z}$ lies  in $W_{-2}(H^{\ab ,\pro\ell})$. Therefore, we obtain  $\overline{z}\in\langle\mathcal{I}_{H^{\ab ,\pro\ell}}\rangle\subset H^{\ab ,\pro\ell}$. $\langle\mathcal{I}_{H^{\ab ,\pro\ell}}\rangle$ is mapped to $\langle \mathcal{I}_{\overline{\Pi}^{m+1}}\cap H_{1}\rangle\mod{ [H_{1},H_{1}]{H_{1}}^{\ell\text{-th}}}$  by the projection $H^{\ab ,\pro\ell}\twoheadrightarrow H_{1}^{\ab ,\pro\ell}\twoheadrightarrow H_1/ [H_{1},H_{1}]{H_{1}}^{\ell\text{-th}}$. Thus, we get  $z_{1}\in [H_{1},H_{1}]{H_{1}}^{\ell\text{-th}}\langle \mathcal{I}_{\overline{\Pi}^{m+1}}\cap H_{1}\rangle$. (Note that this holds trivially even for  $\ell\notin\text{prime}(C)$.) Considering  all $H_{1}$ and primes $\ell$,  we obtain $z_{1}\in \mathcal{I}_{\overline{\Pi}^{m+1}}\cdot \overline{\Pi}^{[m]}/\overline{\Pi}^{[m+1]}$ by Lemma \ref{1.2.6} and Lemma \ref{1.2.5}. \par
Let $z_{0}$ be the image of $z$ by $\overline{\Pi}^{m+2}\twoheadrightarrow\overline{\Pi}^{m}$. By $z_{1}\in \mathcal{I}_{\overline{\Pi}^{m+1}}\cdot \overline{\Pi}^{[m]}/\overline{\Pi}^{[m+1]}$,  there exists an inertia group $\tilde{I}\subset \mathcal{I}_{\overline{\Pi}^{m}}$ that contains $\langle z_{0}\rangle$. By Definition \ref{def1.4.1}(i)(ii), we have $\langle z_{0}\rangle\equiv \tilde{I}\mod{\overline{\Pi}^{[1]}/\overline{\Pi}^{[m]}}\subset \overline{\Pi}^{\ab}$. Since $\tilde{I}$ mapped injectively  by $\overline{\Pi}^m\twoheadrightarrow \overline{\Pi}^{\ab}$, we obtain $\langle z_{0}\rangle=\tilde{I}$. As $I=\langle z_{0}\rangle$, $I$ coincides with  the inertia group $\tilde{I}$, as desired.\par
Finally, we show $(a)\Rightarrow (b)$. Let $y\in\mathscr{S}^{m}(U_{k^{\sep}})$ with  $I=I_y$ and $\tilde{y}$ an inverse image of $y$ by $\mathscr{S}^{m+2}(U_{k^{\sep}})\twoheadrightarrow\mathscr{S}^{m}(U_{k^{\sep}})$.  Inertia groups are maximal cyclic subgroups of cyclotomic type by the assumption $r\geq 2$ and Proposition \ref{1.2.3}. Hence $I_{\tilde{y}}\subset \overline{\Pi}^{m+2}$ satisfies the desired property.
\end{proof}

\begin{prop}\label{1.4.3}
Assume that   $|\text{prime}(C)|=1$, $r\geq 2$ and $k$ is  finitely generated over the prime field. Then for any subgroup $I$ of  $\overline{\Pi}^m$, the following conditions are equivalent.
\begin{enumerate}[(a)]
\item $I$ is an inertia group.
\item $I$ is the image of a   maximal cyclic subgroup of cyclotomic type  of $\overline{\Pi}^{m+1}$  by the map $\overline{\Pi}^{m+1}\rightarrow \overline{\Pi}^{m}$.
\end{enumerate}
\end{prop}
\begin{proof}
Set  $\text{prime}(C)=\left\{\ell\right\}$. First, we show (b)$\Rightarrow $(a). Let  $J$ be a maximal cyclic subgroup of cyclotomic type of $\overline{\Pi}^{m+1}$ whose image by $\overline{\Pi}^{m+1}\rightarrow \overline{\Pi}^{m}$  coincides with $I$. Let $z$ be a generator of $J$ (cf. Definition \ref{def1.4.1}(i)), $z_{0}$ the image of  $z$ by $\overline{\Pi}^{m+1}\twoheadrightarrow\overline{\Pi}^{m}$ and $H_{0}$ an open subgroup of $\overline{\Pi}^{m}$ that contains $z_{0}$.\par
Let $H\subset \overline{\Pi}^{m+1}$ and $\tilde{H}\subset \overline{\Pi}$  be the inverse images of $H_{0}$ by $\overline{\Pi}^{m+1}\twoheadrightarrow\overline{\Pi}^{m}$ and $\overline{\Pi}\twoheadrightarrow \overline{\Pi}^{m}$, respectively.
Since  we have $\overline{\Pi}^{[m]}/\overline{\Pi}^{[m+1]}\subset H$ by definition, we get $\tilde{H}^{\ab }=H^{\ab }$.  We have $\langle\mathcal{I}_{\tilde{H}^{\ab }}\rangle=W_{-2}(\tilde{H}^{\ab })$ by Proposition \ref{1.3.6}, hence  $\langle\mathcal{I}_{H^{\ab }}\rangle = W_{-2}(H^{\ab })$.\par
Let $\overline{z}$ be the image of $z$ by $H\twoheadrightarrow H^{\ab }$.  Since the action of  $p_{U/k}(N_{\Pi^{(m+1)}}(J))$ on $J$ is cyclotomic  by Definition \ref{def1.4.1} (iii)(iv), the action has weight $-2$.  Thus,  $\overline{z}$ lies  in $W_{-2}(H^{\ab })$. Therefore, we obtain  $\overline{z}\in\langle\mathcal{I}_{H^{\ab }}\rangle\subset H^{\ab }$. $\langle\mathcal{I}_{H^{\ab }}\rangle$ is mapped to $\langle \mathcal{I}_{\overline{\Pi}^{m}}\cap H_{0}\rangle\mod{ [H_{0},H_{0}]{H_{0}}^{\ell\text{-th}}}$  by the projection $H^{\ab }\twoheadrightarrow H_{0}^{\ab }\twoheadrightarrow H_0/ [H_{0},H_{0}]{H_{0}}^{\ell\text{-th}}$. Thus, we get  $z_{0}\in [H_{0},H_{0}]{H_{0}}^{\ell\text{-th}}\langle \mathcal{I}_{\overline{\Pi}^{m}}\cap H_{0}\rangle$. (Note that this holds trivially even for  primes different from $\ell$.) Considering  all $H_{0}$ and all primes,  we obtain $z_{0}\in \mathcal{I}_{\overline{\Pi}^{m}}$ by Lemma \ref{1.2.6} and Lemma \ref{1.2.5}. Hence,  there exists an inertia group $\tilde{I}\subset \mathcal{I}_{\overline{\Pi}^{m}}$ that contains $\langle z_{0}\rangle$. By Definition \ref{def1.4.1}(i)(ii), we have $\langle z_{0}\rangle\equiv \tilde{I} \mod{\overline{\Pi}^{[1]}/\overline{\Pi}^{[m]}}\subset \overline{\Pi}^{\ab}$. Since $\tilde{I}$ mapped injectively  by $\overline{\Pi}^m\twoheadrightarrow \overline{\Pi}^{\ab}$, we obtain $\langle z_{0}\rangle=\tilde{I}$. As $I=\langle z_{0}\rangle$, $I$ coincides with  the inertia group $\tilde{I}$, as desired.\par
Finally, we show $(a)\Rightarrow (b)$. Let $y\in\mathscr{S}^{m}(U_{k^{\sep}})$ with  $I=I_y$ and $\tilde{y}$ an inverse image of $y$ by $\mathscr{S}^{m+1}(U_{k^{\sep}})\twoheadrightarrow\mathscr{S}^{m}(U_{k^{\sep}})$.  Inertia groups are maximal cyclic subgroups of cyclotomic type by the assumption $r\geq 2$ and Proposition \ref{1.2.3}. Hence $I_{\tilde{y}}\subset \overline{\Pi}^{m+1}$ satisfies the desired property.

\end{proof}

\begin{prop}\label{1.4.4}
Assume  that $r\geq 3$ and  $k$ is a field finitely generated over the prime field.  Then for any subgroup $D$ of  $\Pi^{(1)}$, the following conditions are equivalent.
\begin{enumerate}[(a)]
\item $D$ is  a decomposition group at cusp.
\item There exists a maximal cyclic subgroup of cyclotomic type  $J$ of $\overline{\Pi}^{3}$ such that  $\Psi^{3}_{\pro\ell,3}(J)$ is a maximal cyclic subgroup of cyclotomic type of $\overline{\Pi}^{\pro\ell,3}$ for all $\ell\in \text{prime}(C)$ and the image of $N_{\Pi^{(\text{nil},2)}}(\Psi^{3}_{\nil,2}(J))$ by $\Psi^{\text{nil},2}_{\nil,1}: \Pi^{(\text{nil},2)}\rightarrow \Pi^{(\text{nil},1)}=\Pi^{(1)}$ coincides with $D$.
\end{enumerate}
\end{prop}
\begin{rem}
Even if $J$ is a maximal cyclic subgroup of cyclotomic type  of $\overline{\Pi}^{3}$,  $\Psi^{3}_{\pro\ell,3}(J)$ may not be a maximal cyclic subgroup of cyclotomic type of $\overline{\Pi}^{\pro\ell,3}$ because $\Psi^{3}_{\pro\ell,3}(J)$ may not satisfy Definition \ref{def1.4.1}(iv).
\end{rem}

\begin{proof}[Proof of \textrm{Proposition 1.4.6}]
(a)$\Rightarrow $(b) follows from the fact that the inertia groups are maximal cyclic subgroups of cyclotomic type (by Proposition \ref{1.2.3}) and Proposition \ref{1.2.7}(3). \par
We consider (b)$\Rightarrow$ (a).  By assumption, \  there exists  $y\in\mathscr{S}^{1}$ and $z_{\ell}\in\mathscr{S}^{\text{pro-}\ell,2}$ such that $\Psi^{3}_{1}(J)=I_{y}$ and $\Psi^{\pro\ell,3}_{\pro\ell,2}(\Psi^{3}_{\pro\ell,3}(J))=I_{z_{\ell}}$ by Proposition \ref{1.4.2} and Proposition \ref{1.4.3}, respectively. Since $\Psi_{\pro\ell,1}^{1}(I_{y})=\Psi_{\pro\ell,1}^{3}(J)=\Psi^{\pro\ell,2}_{\pro\ell,1}(I_{z_{\ell}})$ and $r\neq 2$, Lemma \ref{inersep1}(c)$\Rightarrow$(a) implies  that $y$ and $z_{\ell}$ are mapped to the  same point in $E_{k^{\text{sep}}}$ for all $\ell\in\text{prime}(C)$.\par
By Proposition \ref{1.2.7}(1),  we obtain $z\in \mathscr{S}^{\text{nil},2}$ with $z_{\ell}=\psi^{\nil,2}_{\pro\ell,2}(z)$ for every $\ell\in \text{prime}(C)$. Since  $\overline{\Pi}^{\text{nil},2}$ (resp.  $I_{z}$) is equal to the product of $\overline{\Pi}^{\text{pro-}\ell,2}$ (resp.  $\Psi^{\nil,2}_{\pro\ell,2}(I_{z})$) for all $\ell\in \text{prime}(C)$,  we obtain $\Psi^{3}_{\text{nil},2}(J)=I_{z}$. Therefore,  $\Psi^{\text{nil},2}_{\nil,1}(N_{\Pi^{(\text{nil},2)}}(\Psi^{3}_{\nil,2}(J)))$ is  the decomposition group of $\psi^{\nil,2}_{\nil,1}(z)\in\mathscr{S}^{1}$ by Proposition \ref{1.2.7}(3).
\end{proof}

Let $i=1,2$.  Let $X_{i}$ be a proper smooth curve over $k$  and $E_{i}$ a closed subscheme of $X_{i}$ which is finite \'{e}tale over $k$. Set $U_{i}:=X_{i}-E_{i}$.   Let $g_{i}:=g(U_{i})$ be the  genus of $X_{i,\overline{k}}$ and $r_{i}:=r(U_{i}):=|E_{i,\overline{k}}|$. From Proposition \ref{1.4.2}, Proposition \ref{1.4.3}, Proposition \ref{1.4.4} and Proposition \ref{1.2.3}, we obtain the following corollary.

\begin{cor}\label{1.4.5}
Assume that $k$ is  a field finitely generated over the prime field.  Let $n\in \mathbb{N}$.  Then the following hold. 

\begin{enumerate}[(1)]
\item Assume that $r_{1}\geq2$, $n\geq 2$, and that either $m\geq 2$ or  $r_{1}\geq 3$. Then for any isomorphism $\alpha_{m+n}:\Pi^{(m+n)}(U_{1})\xrightarrow[G_{k}]{\sim}\Pi^{(m+n)}(U_{2})$, the isomorphism $\alpha_{m}:\Pi^{(m)}(U_{1})\xrightarrow[G_{k}]{\sim}\Pi^{(m)}(U_{2})$ induced by $\alpha_{m+n}$  preserves the decomposition groups at cusps.
\item Assume that   $|\text{prime}(C)|=1$ and $r_{1}\geq2$. Then, for any isomorphism $\alpha_{m+n}:\Pi^{(m+n)}(U_{1})\xrightarrow[G_{k}]{\sim}\Pi^{(m+n)}(U_{2})$, the isomorphism $\alpha_{m}:\Pi^{(m)}(U_{1})\xrightarrow[G_{k}]{\sim}\Pi^{(m)}(U_{2})$ induced by $\alpha_{m+n}$ preserves the  inertia groups. Moreover, it preserves the decomposition groups at cusps if  $m\geq 2$.
\end{enumerate}
\end{cor}

\begin{proof}
If $(g_{1},r_{1})=(0,2)$, then the assertions clearly hold. Hence, we may assume that $(g_{1},r_{1})\neq(0,2)$. If $m\geq 2$, (1) follows from Proposition \ref{1.4.2} and Proposition \ref{1.2.3}. If $m=1$, (1) follows from Proposition \ref{1.4.4}. Thus, (1) holds. (2) follows from Proposition \ref{1.4.3} and Proposition \ref{1.2.3}.
\end{proof}

Assume that $r_{1}\geq3$ and $k$ is  a field finitely generated over the prime field.  Let $\text{Isom}_{G_{k}}^{\text{Iner}}(\Pi^{(m)}(U_{1}),\Pi^{(m)}(U_{2}))$ (resp. $\text{Isom}_{G_{k}}^{\text{Dec}}(\Pi^{(m)}(U_{1}),\Pi^{(m)}(U_{2}))$) be  the set of all $G_{k}$-isomorphism $\Pi^{(m)}(U_{1})\xrightarrow{\sim}\Pi^{(m)}(U_{2})$ which preserve inertia groups (resp. decomposition groups at cusps). Let   $\text{Isom}_{G_{k}}(E_{1,k^{\text{sep}}},E_{2,k^\text{sep}})$ be the set of all bijection which  compatible with the $G_{k}$ actions on $E_{1,k^\text{sep}}$ and $E_{2,k^\text{sep}}$. Then we have  the map 
\begin{equation}\label{eqwr}
\Phi'_{I}: \text{Isom}_{G_{k}}^{\text{Iner}}(\Pi^{(m)}(U_{1}),\Pi^{(m)}(U_{2}))\rightarrow \text{Isom}_{G_{k}}(E_{1,k^{\text{sep}}},E_{2,k^\text{sep}})
\end{equation}
by Lemma \ref{inersep1} and Lemma \ref{inersepm},  since $\Phi'_{I}$ is defined as 
\begin{equation*}\label{wq1.4.10}
\text{$\Phi'_{I}(\alpha)(x_{1}):=x_{2}$ when   $\alpha_{1}(\mathcal{I}_{x_{1},\overline{\Pi}^1})=\mathcal{I}_{x_{2},\overline{\Pi}^1}$}
\end{equation*}
for any $\alpha\in\text{Isom}_{G_{k}}(\Pi^{(m)}(U_{1}),\Pi^{(m)}(U_{2}))$,  $x_{1}\in E_{1,k^{\text{sep}}}$, and $x_{2}\in E_{2,k^{\text{sep}}}$,   where $\alpha_{1}$ stands for the element in $\text{Isom}_{G_{k}}(\Pi^{(1)}(U_{1}),\Pi^{(1)}(U_{2}))$ induced by  the image of $\alpha$. We know that $\text{Isom}_{G_{k}}^{\text{Dec}}(\Pi^{(m)}(U_{1}),\Pi^{(m)}(U_{2}))\subset \text{Isom}_{G_{k}}^{\text{Iner}}(\Pi^{(m)}(U_{1}),\Pi^{(m)}(U_{2}))$, hence $\Phi'$ induces the map
\begin{equation}\label{eq1.1.9}
\Phi'_{D}: \text{Isom}_{G_{k}}^{\text{Dec}}(\Pi^{(m)}(U_{1}),\Pi^{(m)}(U_{2}))\rightarrow \text{Isom}_{G_{k}}(E_{1,k^{\text{sep}}},E_{2,k^\text{sep}}).
\end{equation}
Let $n\in\mathbb{N}$. By Corollary \ref{1.4.5}, we have that the image of the map 
\begin{equation*}
\Phi'':\text{Isom}_{G_{k}}(\Pi^{(m+n)}(U_{1}),\Pi^{(m+n)}(U_{2}))\rightarrow\text{Isom}_{G_{k}}(\Pi^{(m)}(U_{1}),\Pi^{(m)}(U_{2}))
\end{equation*}
 is contained in $\text{Isom}^{\text{Iner}}_{G_{k}}(\Pi^{(m)}(U_{1}),\Pi^{(m)}(U_{2}))$  when $n\geq 2$ or   $|\text{prime}(C)|=1$, and is contained in  $\text{Isom}^{\text{Dec}}_{G_{k}}(\Pi^{(m)}(U_{1}),$ \\$\Pi^{(m)}(U_{2}))$  when ``$n\geq 2$'' or ``$|\text{prime}(C)|=1$ and $m\geq 2$''. Thus,  we obtain the map 
\begin{equation}\label{wq1.4.9}
\Phi:=\Phi^{m+n}:=\Phi'_{I}\circ \Phi'':\text{Isom}_{G_{k}}(\Pi^{(m+n)}(U_{1}),\Pi^{(m+n)}(U_{2}))\rightarrow \text{Isom}_{G_{k}}(E_{1,k^{\text{sep}}},E_{2,k^\text{sep}}).
\end{equation}
when $n\geq 2$ or   $|\text{prime}(C)|=1$. \par

%%%%%%%%%%%%%%%%%%%%%%%%%%%%%%%%%%%%%%%%%%%%%%%%%%%%%%%%%%%%%%%%%%%%%%%%%%%%%%%%%%%%%%%%%%%%%
%%%%%%%%%%%%%%%%%%%%%%%%%%%%%%%%%%%%%%%%%%%%%%%%%%%%%%%%%%%%%%%%%%%%%%%%%%%%%%%%%%%%%%%%%%%%%
%%%%%%%%%%%%%%%%%%%%%%%%%%%%%%%%%%%%%%%%%%%%%%%%%%%%%%%%%%%%%%%%%%%%%%%%%%%%%%%%%%%%%%%%%%%%%

\section{The  $m$-step solvable Grothendieck conjecture for genus 0 curves}\label{section2}
\hspace{\parindent}
We continue to use  Notation $1$ and  Notation $2$. Assume that  $k$ is a  field finitely generated over the prime field. Assume that   $U_{1}$ and $U_{2}$ are genus $0$ hyperbolic curves. In this section, we  reconstruct an isomorphism $U_{1}\underset{k}\cong U_{2}$ from a given  isomorphism $\Pi^{(1)}(U_{1})\underset{G_k}\cong \Pi^{(1)}(U_{2})$ which preserves the decomposition groups at cusps  (under certain conditions).  Since we have already reconstructed  an isomorphism $\Pi^{(1)}(U_{1}) \underset{G_{k}}\cong \Pi^{(1)}(U_{2})$ which preserves the decomposition groups at cusps from a given isomorphism $\Pi^{(m)}(U_{1})\underset{k}\cong \Pi^{(m)}(U_{2})$ when  $m\geq 3$, this implies the main theorem. In subsection $2.1$, we define the rigidity invariants  and  show some facts in field theory. In subsections $2.2$ and $2.3$, we show the $m$-step solvable Grothendieck conjecture for punctured projective lines over $k$ of characteristic $0$ and positive characteristic, respectively.  In subsection $2.4$, we show the main theorem by Galois descent.\par
This section mainly refers to sections $4$ and $6$ of  \cite{Na1990-411}.

%%%%%%%%%%%%%%%%%%%%%%%%%%%%%%%%%%%%%%%%%%%%%%%%%%%%%%%%%%%%%%%%%%%%%%%%%%%%%

\subsection{Rigidity invariant}\label{subsection2.1}

\hspace{\parindent} In this subsection, we define the rigidity invariant for  $\Pi^{(1)}(U)$. The rigidity invariant is defined in  \cite{Na1990-411}(4.2) in the case of the full fundamental group and the following  definition is essentially the same as in \cite{Na1990-411}(4.2).

\begin{definition}\label{2.1.1}
Assume that $U=\mathbb{P}^1_k-E$, where $E$ is a finite set of $k$-rational points of $\mathbb{P}_k^1$ with $|E|\geq 4$. Let $n\in\mathbb{N}(C)$.   Let $x_1,x_2,x_3,x_4$ be distinct elements of $E$ and $\varepsilon=\left\{x_1,x_2\right\},\delta=\left\{x_3,x_4\right\}$.
\begin{enumerate}[(1)]
\item We denote by $\mathscr{H}_{\varepsilon,n}$  the set of  all open subgroups $H$ of  $\Pi^{(1)}(:=\Pi^{(1)}(U))$ that satisfy the following conditions.
\begin{enumerate}[(i)]
\item $\overline{H}:=H\cap \overline{\Pi}^1$ contains $\mathcal{I}_{x,\overline{\Pi}^1}$ for all $x\in S-\varepsilon$.
\item $\overline{\Pi}^1/\overline{H} \cong \mathbb{Z}/n\mathbb{Z}$
\item $p_{U/k}(H)=G_{k(\mu_n)}$
\item $p_{U/k}^{-1}(G_{k(\mu_n)})\rhd H$
\end{enumerate}

\item   We define  $\kappa_n(\varepsilon,\delta)$ to be the subfield of $k^{\sep}$ consisting of  the elements fixed by all the automorphisms belonging to 
\begin{equation*}
\bigcup_{ H\in \mathscr{H}_{\varepsilon,n}}\bigcap_{y\in\mathscr{S}^1_{x_3}\cup \mathscr{S}^1_{x_4}}p_{U/k}(H\cap D_y).
\end{equation*}
We call  $\kappa_n(\varepsilon,\delta)$ the rigidity invariant for $\varepsilon, \delta$ of $U$.  
\end{enumerate}
\end{definition}

Set the following notation for $\varepsilon$ and $\delta$.
\[
\lambda  (\varepsilon, \delta):=\frac{ x_4-x_1 }{x_4-x_2} \frac{ x_3-x_2 }{x_3-x_1}
\]
By definition,  the isomorphism $\mathbb{P}^1_k\xrightarrow[k]{\sim} \mathbb{P}^1_k$ satisfying $x_1\mapsto 0,\ x_2\mapsto \infty,\ x_3\mapsto 1$ maps $x_4$  to $\lambda  (x_1,x_2,x_3,x_4)$.\par
The next proposition is  essentially the same as \cite{Na1990-411}(4.3).  Considering the difference between Definition \ref {2.1.1} and the definition given in \cite{Na1990-411}(4.2), we give a proof again here.

\begin{prop}\label{2.1.3} Under the notation of Definition \ref{2.1.1}, the following hold.
\begin{equation*}
\kappa_n(\varepsilon,\delta)\ =\ k(\mu_n,\lambda(\varepsilon,\delta)^{\frac{1}{n}})\ \ \ \ \ \ \ \ \ \ (n\in\mathbb{N}(C))
\end{equation*}
\end{prop}

\begin{proof}
Let $t:\mathbb{P}^1_k\xrightarrow[k]{\sim} \mathbb{P}^1_k$ be the isomorphism that satisfies $t(x_1)=0, t(x_2)=\infty$ and $t(x_3)=1$. Then we have  $t(x_4)=\lambda(\varepsilon,\delta)$. \par
Let  $H\in\mathscr{H}_{\varepsilon,n}$ and let $U^{\overline{H}}\rightarrow U_{k^{\sep}}$ be the cover corresponding to $\overline{H}\subset \overline{\Pi}^1$. Then $\overline{\Pi}^1/\overline{H}\cong \mathbb{Z}/n\mathbb{Z}$ and $\overline{H}$ contains $\mathcal{I}_{x,\overline{\Pi}^1}$ for all $x\in S-\varepsilon$ by Definition \ref{2.1.1}(1)(ii) and (i), hence  $(U^{\overline{H}})^{\text{cpt}}\rightarrow (U_{k^{\sep}})^{\text{cpt}}$ corresponds to a unique cover of degree $n$ unramified outside $t(\varepsilon)=\left\{0,\infty\right\}$. Thus, $(U^{\overline{H}})^{\text{cpt}}\rightarrow (U_{k^{\sep}})^{\text{cpt}}$ is identified with $\mathbb{P}^1_k\rightarrow\mathbb{P}^1_k$, $x\mapsto x^n$ and corresponding extension of function fields is $k^{\sep}(t^{\frac{1}{n}})/k^{\sep}(t)$.\par
Let  $U^{H}\rightarrow U_{k(\mu_n)}$ be the cover  corresponding to  $H \lhd p_{U/k}^{-1}(G_{k(\mu_n)})$ (cf. Definition \ref{2.1.1}(1)(iv)). Since Gal$(K(U^{H})/k(\mu_n,t))\cong p_{U/k}^{-1}(G_{k(\mu_n)})/H\cong \mathbb{Z}/n\mathbb{Z}$ by Definition \ref{2.1.1}(1)(ii)(ii) and (iv),  there exists a subgroup $\Delta\subset k(\mu_n,t)^{\times}/k(\mu_n,t)^{\times n}$ such that  $K(U^{H})\cong k(\mu_n,\Delta^{\frac{1}{n}})$ and $\Delta\cong \text{Gal}(K(U^{H})/k(\mu_n,t))\cong \mathbb{Z}/n\mathbb{Z}$ by Kummer theory. Since $\Delta$ is mapped isomorphically to $\langle t\rangle\mod{k^{\sep}(t)^{\times n}}$ by $k(\mu_n,t)^{\times}/k(\mu_n,t)^{\times n}\rightarrow  k^{\sep}(t)^{\times}/k^{\sep}(t)^{\times n}$, there exists $f\in k^{\sep}(t)^{\times}$ such that $tf^{n}\in k(\mu_{n},t)^{\times}$ and $\Delta=\langle tf^{n}\rangle\mod{k(\mu_n,t)}$. Thus, we get $K(U^{H})=k(\mu_n,(tf^n)^{\frac{1}{n}})$. Since $t$, $tf^n\in k(\mu_n,t)^{\times}$, we have $f^n\in k^{\sep}(t)^{\times n}\cap k(\mu_n,t)^{\times}=k(\mu_n)^{\times}\cdot k(\mu_n,t)^{\times n}$, hence we get the following consequence.
\begin{equation}\label{eq2.1}
\text{ There exists }\omega_H\in k(\mu_n)^{\times}\text{ such that }K(U^{H})=k(\mu_n,(\omega_{H}t)^{\frac{1}{n}}).
\end{equation}\par
Let $\kappa_H\subset k^{\sep}$ be the  fixed field by $\bigcap_{y\in\mathscr{S}^1_{x_3}\cup\mathscr{S}^1_{x_4}} p_{U/k}(H\cap D_y)$. Let $y\in\mathscr{S}^1_{x_3}\cup\mathscr{S}^1_{x_4}$ and write $x$ for the image of $y$ in  $(U^{H})^{\text{cpt}}$. Then we have $p_{U/k}(H\cap D_y)$=$G_{\kappa(x)}$, hence $\kappa_H$ is the composite field of  residue fields of  all $x$ above $x_3$ and $x_4$. By (\ref{eq2.1}), we get 
\begin{equation}\label{eq2.2}
\kappa_H=k(\mu_n,\left\{\omega_H\right\}^{\frac{1}{n}},\left\{\omega_H\lambda(\varepsilon,\delta)\right\}^{\frac{1}{n}}).
\end{equation}\par
Let $H_0\in\mathscr{H}_{\varepsilon,n}$ that satisfies $K(U^{H_0})=k(\mu_n,t^{\frac{1}{n}})$. Since  $\omega_{H_0}=1$, (\ref{eq2.2}) implies  $\kappa_{H_0}=k(\mu_n,\lambda(\varepsilon,\delta)^{\frac{1}{n}})$.  Clearly $\kappa_H=k(\mu_n,\left\{\omega_H\right\}^{\frac{1}{n}},\left\{\omega_H\lambda(\varepsilon,\delta)\right\}^{\frac{1}{n}})\supset k(\mu_n,\lambda(\varepsilon,\delta)^{\frac{1}{n}})=\kappa_{H_0}$ for all $H\in\mathscr{H}_{\varepsilon,n}$, hence we obtain $\kappa_n(\varepsilon,\delta)=\underset{H}\bigcap\kappa_H=\kappa_{H_0}=k(\mu_n,\lambda(\varepsilon,\delta)^{\frac{1}{n}})$.
\end{proof}

Next,  we show some facts in field theory.

\begin{lem}\label{2.1.4}
Let $\ell$ be a prime, $n\in\mathbb{N}$ and set $N:=\ell^n$.  If $4\nmid N$ or $k\supset \mu_4$, then $k(\mu_N)^{\times N}\cap k^{\times}=k^{\times N}$ holds.
\end{lem}

\begin{proof}
If $\ell= p$, then $\mu_N=\left\{1\right\}$ and $k(\mu_N)^{\times N}\cap k^{\times}=k^{\times N}$. If $\ell\neq p$, set $G:=\text{Gal}(k(\mu_N)/k)$. By  the exact sequence $1\rightarrow \mu_N\rightarrow k(\mu_N)^{\times}\rightarrow k(\mu_N)^{\times N}\rightarrow 1$ of $G$-modules, we get the following long exact sequence.
\[
\xymatrix@R=16pt@C=40pt{
H^{0}(G,k(\mu_N)^{\times})\ar[r]^-{N\text{-th power}}\ar@{}[d]|*=0[@]{\rotatebox{90}{$\cong$}}	&H^{0}(G,k(\mu_N)^{\times N})\ar[r]\ar@{}[d]|*=0[@]{\rotatebox{90}{$\cong$}}		&H^{1}(G,\mu_N)\ar[r]		&H^{1}(G,k(\mu_N)^{\times})\ar@{}[d]|*=0[@]{\rotatebox{90}{$\cong$}}^{\ \ \text{Hilbert's theorem 90}} \\
k^{\times}									&k(\mu_N)^{\times N}\cap k^{\times}					&&0
}
\]
Thus, $H^{1}(G,\mu_N)\cong (k(\mu_N)^{\times N}\cap k^{\times})/k^{\times N}$ and it is  sufficient to show that $H^{1}(G,\mu_N)=0$.  If $|G|$ is not divided by $\ell$, then $(|G|,|\mu_{N}|)=1$ and  this assertion  follows from a general theory of  group cohomology. Hence, we may assume that $|G|$ is divided by $\ell$.\par
We fix an isomorphism  $\mu_{N}\overset{\sim}{\rightarrow }\mathbb{Z}/N\mathbb{Z}$ and regard $G$ as a subgroup of $(\mathbb{Z}/N\mathbb{Z})^{\times}$. Note that $(\mathbb{Z}/N\mathbb{Z})^{\times}$ is a cyclic group by $N=\ell^n$ (and  $G \subset \text{Ker}((\mathbb{Z}/N\mathbb{Z})^{\times}\twoheadrightarrow(\mathbb{Z}/4\mathbb{Z})^{\times})$ by  $k\supset \mu_4$ if $4|N$). Let $\sigma\in G$ be a  generator. Then we get the following equality from the calculation of group cohomology.

\[
H^1(G,\mathbb{Z}/N\mathbb{Z})=\set{\alpha\in\mathbb{Z}/N\mathbb{Z}}{(1+\sigma+\cdots+\sigma^{|G|-1})\alpha=0}/((\sigma-1)\mathbb{Z}/N\mathbb{Z})
\]
Let $\tilde{\sigma}\in \mathbb{Z}$ be an inverse image of $\sigma$ by $\mathbb{Z}\rightarrow \mathbb{Z}/N\mathbb{Z}$ and  set $Y:=\langle \tilde{\sigma}\mod{\ell^{n+1}}\rangle\subset(\mathbb{Z}/\ell^{n+1}\mathbb{Z})^{\times}$. Then $Y$ is mapped surjectively onto $G$ by $(\mathbb{Z}/\ell^{n+1}\mathbb{Z})^{\times}\twoheadrightarrow (\mathbb{Z}/N\mathbb{Z})^{\times}$. Moreover, $Y$ includes the kernel of $(\mathbb{Z}/\ell^{n+1}\mathbb{Z})^{\times}\rightarrow(\mathbb{Z}/N\mathbb{Z})^{\times}$ because $\ell\mid |Y|$ (and  $G \subset \text{Ker}((\mathbb{Z}/N\mathbb{Z})^{\times}\twoheadrightarrow(\mathbb{Z}/4\mathbb{Z})^{\times})$ if $4|N$). Then $|G|<|Y|$. Hence, $\ell^{n+1}\nmid(\tilde{\sigma}^{|G|}-1)$ and $\ell^{n}|(\tilde{\sigma}^{|G|}-1)$. Thus, we get $\text{ord}_\ell(\tilde{\sigma}-1)+\text{ord}_{\ell}(1+\tilde{\sigma}+\cdots+\tilde{\sigma}^{|G|-1})=n$. 
\[
\left|\set{\alpha\in\mathbb{Z}/N\mathbb{Z}}{(1+\sigma+\cdots+\sigma^{|G|-1})\alpha=0}\right|=\ell^{\text{ord}_{\ell}(1+\tilde{\sigma}+\cdots+\tilde{\sigma}^{|G|-1})}=\ell^{n-\text{ord}_\ell(\tilde{\sigma}-1)}=|(\sigma-1)\mathbb{Z}/N\mathbb{Z}|
\]
This implies  $H^1(G,\mathbb{Z}/N\mathbb{Z})=0$, as desired.
\end{proof}

\begin{prop}\label{2.1.5}
Assume that $k$ is finitely generated over the prime field. Let $i=1,2$.  Let $\mathbb{T}$ be a set of  primes that differ from $p$  and $\text{Pw}(\mathbb{T})\subset \mathbb{N}$ the set of all powers of primes in $\mathbb{T}$. Let  $\Gamma_{i}$  be finitely generated subgroups of  $k^{\times}$. Denote $\Gamma_{i}^{\frac{1}{n}}$  by the set of all elements which $n$-th power is contained in $\Gamma_{i}$. If  $k(\Gamma_{1}^{\frac{1}{n}})=k(\Gamma_{2}^{\frac{1}{n}}) \text{ for all }n\in \text{Pw}(\mathbb{T})$, then there exist $N_{1}, N_{2}\in \mathbb{N}$ which is prime to all elements of $\mathbb{T}$ such that $\Gamma_{1}^{N_{1}}\subset \Gamma_{2}$ and $\Gamma_{1}\supset \Gamma_{2}^{N_{2}}$ hold. \par
\end{prop}

\begin{proof}(See \cite{Na1990-405}Lemma 3.1.)
First, we show the assertion. We may  replace $k$ with a field finitely generated over $k$, hence we may assume that $k\supset \mu_4$ if $2\in \mathbb{T}$.  If $k(\Gamma_{1}^{\frac{1}{n}})=k(\Gamma_{2}^{\frac{1}{n}})$, then $k(\Gamma_{1}^{\frac{1}{n}})\subset k((\Gamma_{1}\Gamma_{2})^{\frac{1}{n}})\subset k(\Gamma_{1}^{\frac{1}{n}})\cdot k(\Gamma_{2}^{\frac{1}{n}})=k(\Gamma_{1}^{\frac{1}{n}})$ and $k(\Gamma_{1}^{\frac{1}{n}})=k((\Gamma_{1}\Gamma_{2})^{\frac{1}{n}})$. Thus,  replacing $\Gamma_{2}$ with $\Gamma_{1}\Gamma_{2}$ if necessary, we may assume $\Gamma_{1}\subset \Gamma_{2}$.\par
 Fix an element $\gamma_{2}\in \Gamma_{2}$ and an element $n\in \text{Pw}(\mathbb{T})$.  We have $k(\mu_{n},\gamma_{2}^{\frac{1}{n}})\subset k(\Gamma_{2}^{\frac{1}{n}})=k(\Gamma_{1}^{\frac{1}{n}})$ by assumption. By Kummer correspondence, we get  $\langle \gamma_{2}\rangle\mod{ k(\mu_{n})^{\times n}}\ \subset \Gamma_{1}\mod{k(\mu_{n})^{\times n}}$. Hence, $\gamma_{2}\in \Gamma_{1}\cdot k(\mu_{n})^{\times n}$. By Lemma \ref{2.1.4}, we get $\gamma_{2}\in(\Gamma_{1}\cdot k(\mu_{n})^{\times n})\cap k^{\times}=\Gamma_{1}\cdot (k(\mu_{n})^{\times n}\cap k^{\times})=\Gamma_{1} \cdot k^{\times n}$. \par
Let $R$ be the integral closure of $\mathbb{Z}[\Gamma_{2}]$ in $k$ if $p=0$ and  the integral closure of  $\mathbb{F}_p[\Gamma_{2}]$ in $k$ if $p\neq 0$. As $R$ is a finitely generated $\mathbb{Z}$-algebra by  assumption, $R^{\times}$ is a finitely generated $\mathbb{Z}$-module and $R^{\times}/(R^{\times}\cap \Gamma_{1})=\Gamma_{1} R^{\times}/\Gamma_{1}$ is also a finitely generated $\mathbb{Z}$-module. Hence, $\underset{n\in\mathbb{T}}\bigcap (\Gamma_{1} R^{\times}/\Gamma_{1})^{n}$ is a finite group whose order $N_{2}$ is prime to all elements of $\mathbb{T}$. Since $\gamma_{2}\in\Gamma \cdot k^{\times n} \cap R^{\times}$ and $\Gamma_{1}\subset R^{\times}$, we get $\gamma_{2}\in\Gamma \cdot R^{\times n}$.  Then we obtain $\gamma_{2} $ mod $\Gamma_{1} \in  \underset{n\in\text{Pw}(\mathbb{T})}\bigcap (\Gamma_{1} R^{\times}/\Gamma_{1})^{n}$. Hence we get  $\gamma_{2}^{N_{2}}\in  \Gamma_{1}$.  As $\gamma_{2}\in \Gamma_{2}$ is arbitrary, this shows  $\Gamma_{2}^{N_{2}}\subset  \Gamma_{1}$. Thus,  the assertion follows.\par
\end{proof}

\begin{cor}\label{2.1.6}
Assume that $k$ is finitely generated over the prime field. Let $\lambda_{1}$, $\lambda_{2} \in k^{\times}$. If $k(\langle\gamma_{1}\rangle^{\frac{1}{\ell^{n}}})=k(\langle\gamma_{2}\rangle^{\frac{1}{\ell^{n}}})$ for all $\ell$ different from $p$ and  all $n\in \mathbb{N}$, then the following hold.
\begin{enumerate}[(1)] 
\item If $p=0$, then $\langle \lambda_{1}\rangle=\langle \lambda_{2}\rangle$.
\item  If $p\neq0$, there exists $ \sigma\in \mathbb{Z}$ such that $\langle\gamma_{1}\rangle^{p^{\sigma}}=\langle \gamma_{2}\rangle$. If, moreover, $\gamma\in k^{\times}$ is not a torsion element, then such $\sigma$ is unique.
\end{enumerate}
\end{cor}

\begin{proof}
(1) Applying Proposition \ref{2.1.5} for all primes and  all $n\in \mathbb{N}$, we get $N=1$ and then $\langle \lambda_{1}\rangle=\langle \lambda_{2}\rangle$.\\
(2) By Proposition \ref{2.1.5}, there exist $u_{1},u_{2}\in \mathbb{N}\cup\left\{0\right\}$ such that $\gamma_{1}^{p^ u_{1}}\in \langle\gamma_{2}\rangle$ and  $\langle\gamma\rangle\ni\gamma_{2}^{p^{u_{2}}}$.
If $\gamma_{1}$ is a torsion element, then $\gamma_{2}$ is also a torsion element. Then, since $p\nmid |\langle\gamma_{1}\rangle|$ and $p\nmid |\langle\gamma_{2}\rangle|$, we get  $\langle\gamma_{1}\rangle=\langle\gamma_{1}\rangle^{p^{u_{1}}}\subset \langle\gamma_{2}\rangle $ and  $\langle\gamma_{1}\rangle\supset\langle\gamma_{2}\rangle^{p^{u_{2}}}=\langle\gamma_{2}\rangle$. Thus, $\langle\gamma_{1}\rangle=\langle\gamma_{2}\rangle$.
If $\gamma_{1}$ is not a torsion element, then there exist $a_{1},a_2\in\mathbb{Z}$ such that $\gamma_{1}^{p^{ u_{1}}}=\gamma_{2}^{a_{2}}$ and $\gamma_{1}^{a_{1}}=\gamma_{2}^{p^{u_{2}}}$. Then $p^{u_{1}+u_{2}}=a_{1}a_{2}$ and  $|a_{1}|,|a_{2}|$ are powers of $p$. Hence, we get $\gamma_{1}^{p^\sigma}=\gamma_{2}$ or $\gamma_{1}^{p^\sigma}=\gamma_{2}^{-1}\ \text{for some }\sigma\in\mathbb{Z}$ and   $\langle\gamma_{1}\rangle^{p^{\sigma}}=\langle \gamma_{2}\rangle$.  The uniqueness of $\sigma$ follows from the fact that  $\gamma_{1}$ is not a torsion element.\par
\end{proof}

%%%%%%%%%%%%%%%%%%%%%%%%%%%%%%%%%%%%%%%%%%%%%%%%%%%%%%%%%%%%%%%%%%%%%%%%%%%%%%%%%%%%%%%%%%%%%%%%%

\subsection{Case of punctured projective lines over fields of characteristic 0 } \label{ch0}

\hspace{\parindent}In this subsection, we show the  $m$-step solvable Grothendieck conjecture for punctured projective lines in characteristic $0$. \par

First, we consider the genus $0$ hyperbolic curves.  Let us  reconstruct $\lambda\in k^{\times}-\left\{1\right\}$ from  $\langle\lambda\rangle$ and $\langle1-\lambda\rangle$ in $k^{\times}$.

\begin{lem}\label{2.2.1}
Assume that $p=0$. Let  $\lambda_{1},\lambda_{2}\in k^{\times}-\left\{1\right\}$ and let $\rho\in \overline{k}$ be a primitive $6$-th root of unity.  If  $\langle\lambda\rangle =\langle \lambda_{2}\rangle$ and  $\langle 1-\lambda_{1}\rangle=\langle 1-\lambda_{2}\rangle$ in $k^\times$, then $\lambda_{1}=\lambda_{2}$ or $\left\{\lambda_{1},\lambda_{2}\right\}=\left\{\rho,\rho^{-1}\right\}$.
\end{lem}

\begin{proof}
By replacing $k$ with $\mathbb{Q}(\lambda_{1},\lambda_{2})$, we may assume that $k$ is a field finitely generated over $\mathbb{Q}$. We fix an embedding  $k\hookrightarrow \mathbb{C}$ and regard  $k$ as a subfield of $\mathbb{C}$.  In particular, we may assume  $\rho=\mathrm{\text{e}}^{\frac{\pi}{3}\sqrt{-1}}$. \par
Suppose  $\lambda_{1}\neq\tilde{\lambda_{1}}$. If $|\lambda_{1}|\neq1$, then $\lambda_{1}$ is a torsion-free element and $\mathbb{Z}\cong\langle\lambda_{1}\rangle=\langle\lambda_{2}\rangle$, hence $\lambda_{1}=\lambda_{2}^{-1}$ follows. Similarly,  if  $|1-\lambda_{1}|\neq1$, then $1-\lambda_{1}=(1-\lambda_{2})^{-1}$ follows.\par
\begin{itemize}
\item If  $|\lambda_{1}|\neq1$ and $|1-\lambda_{1}|\neq1$,   we have $\lambda_{1}=\lambda_{2}^{-1}$ and $1-\lambda_{1}=(1-\lambda_{2})^{-1}$. Then $\lambda_{1},\ \lambda_{2}$ are roots of   $t^2-t+1\in \mathbb{C}[t]$ and $\lambda_{1},\lambda_{2}\in\mu_6$. This is absurd.
\item If  $|\lambda_{1}|\neq1$ and $|1-\lambda_{1}|=1$,  we have  $\lambda_{1}=\lambda_{2}^{-1}$. Then $1=|1-\lambda_{1}|=|1-\lambda_{2}|$ and $(\lambda_{1}-1)=\lambda_{1}(1-\lambda_{2})$, which implies $|\lambda_{1}|=1$. This is absurd.
\item If $|\lambda_{1}|=1$ and $|1-\lambda_{1}|\neq1$,  we have $1-\lambda_{1}=(1-\lambda_{2})^{-1}$. Then $1=|\lambda_{1}|=|\lambda_{2}|$ and $-\lambda_{1}=(1-\lambda_{1})\lambda_{2}$, which implies $|1-\lambda_{1}|=1$. This is absurd.
\item If  $|\lambda_{1}|=1$ and $|1-\lambda_{1}|=1$, then  set  $\lambda_{1}=a+b\sqrt{-1}\ (a,b\in\mathbb{R})$. Since $(a,b),(0,0),(1,0)\in \mathbb{R}^2$ is an equilateral triangle in $\mathbb{R}^{2}$, we get $\left\{\lambda_{1},\lambda_{2}\right\}=\left\{\mathrm{e}^{\frac{\pi}{3}\sqrt{-1}},\mathrm{e}^{-\frac{\pi}{3}\sqrt{-1}}\right\}$.
\end{itemize}
\end{proof}

\begin{prop}\label{2.2.2}
Assume that $ k$ is a field finitely generated over  $\mathbb{Q}$,  and that  $\text{prime}(C)$ coincides with the set of all primes. Let  $\rho\in \overline{k}$ be a primitive $6$-th root of unity. Let   $E_{1}$,$E_{2}$ be  finite sets of $k$-rational points of $\mathbb{P}^1_{k}$ with $|E_{1}|\geq 4$.  . Let $x_1,x_2,x_3,x_4$  be distinct elements of $E_{1}$, $\varepsilon=\left\{x_1,x_2\right\}$ and $\delta=\left\{x_3,x_4\right\}$. Then $\lambda(\varepsilon,\delta)=\lambda(\Phi'(\alpha)(\varepsilon),\Phi'(\alpha)(\delta))\text{ or }\left\{\lambda(\varepsilon,\delta),\ \lambda(\Phi'(\alpha)(\varepsilon),\Phi'(\alpha)(\delta))\right\}=\left\{\rho,\ \rho^{-1}\right\}$ holds for any $\alpha \in \text{Isom}_{G_{k}}(\Pi^{(m)}(\mathbb{P}^1_k-E_{1}),\Pi^{(m)}(\mathbb{P}^1_k-E_{2}))$, where $\Phi':=\Phi'_{D}$ is the map (\ref{eq1.1.9}). 
\end{prop}
\begin{proof}
Let $\alpha \in \text{Isom}_{G_{k}}(\Pi^{(m)}(\mathbb{P}^1_k-E_{1}),\Pi^{(m)}(\mathbb{P}^1_k-E_{2}))$ and  $\alpha_{1}$ the image of $\alpha$ in $\text{Isom}_{G_{k}}^{\text{Dec}}(\Pi^{(1)}(\mathbb{P}^1_k-E_{1}),\Pi^{(1)}(\mathbb{P}^1_k-E_{2}))$.   Since  $\kappa_n(\varepsilon,\delta)$ is characterized   by $\Pi^{(1)}(\mathbb{P}^1_k-E_{1})\twoheadrightarrow G_k$ and the decomposition groups at cusps of $\Pi^{(1)}(\mathbb{P}^1_k-E_{1})$ group-theoretically, we get $k(\mu_n,\lambda(\varepsilon,\delta)^{\frac{1}{n}})=k(\mu_n,\lambda(\Phi'(\alpha)(\varepsilon),\Phi'(\alpha)(\delta))^{\frac{1}{n}})$ for every $n\in\mathbb{N}$ by Proposition \ref{2.1.3}. Hence $ \langle\ \lambda(\varepsilon,\delta)\ \rangle=\langle\ \lambda(\Phi'(\alpha)(\varepsilon),\Phi'(\alpha)(\delta))\ \rangle$ by Corollary \ref{2.1.6} (1).  Set $\varepsilon'=\left\{x_3,x_2\right\},\delta'=\left\{x_1,x_4\right\}$. Then $\lambda(\varepsilon',\delta')=1-\lambda(\varepsilon,\delta)$. Thus, similarly, we get $\langle1-\lambda(\varepsilon,\delta)\rangle=\langle1-\lambda(\Phi'(\alpha)(\varepsilon),\Phi'(\alpha)(\delta))\rangle$. Therefore, the assertion follows from Lemma \ref{2.2.1}.
\end{proof}

We  show the following lemma for the projective line minus $4$ points.

\begin{lem}\label{2.2.3}
Assume that $m\geq 2$,  $ k$ is a field finitely generated over  $\mathbb{Q}$,  and that  $2\in \text{prime}(C)$.    Let  $\rho\in \overline{k}$ be a primitive $6$-th root of unity. Let $W\in \text{Isom}_{\text{set}}(\left\{0,1,\infty,\rho\right\},\left\{0,1,\infty,\rho^{-1}\right\})$ be  the element defined as $W(0)=0$, $W(\infty)=\infty$,$W(1)=1$,  and $W(\rho)=\rho^{-1}$.  Then the image of $\Phi'$ does not contain $W$, where \[\Phi':=\Phi'_{D}:\text{Isom}^{\text{Dec}}_{G_{k}}(\Pi^{(m)}(\mathbb{P}^1_k-\left\{0,1,\infty,\rho\right\}),\Pi^{(m)}(\mathbb{P}^1_k-\left\{0,1,\infty,\rho^{-1}\right\}))\rightarrow \text{Isom}_{\text{set}}(\left\{0,1,\infty,\rho\right\},\left\{0,1,\infty,\rho^{-1}\right\})\]
is  the map (\ref{eq1.1.9}). .
\end{lem}

\begin{proof}
If $\text{Isom}^{\text{Dec}}_{G_{k}}(\Pi^{(m)}(\mathbb{P}^1_k-\left\{0,1,\infty,\rho\right\}),\Pi^{(m)}(\mathbb{P}^1_k-\left\{0,1,\infty,\rho^{-1}\right\}))=\emptyset$, then the assertion is clear. Hence we assume that  $\text{Isom}^{\text{Dec}}_{G_{k}}(\Pi^{(m)}(\mathbb{P}^1_k-\left\{0,1,\infty,\rho\right\}),\Pi^{(m)}(\mathbb{P}^1_k-\left\{0,1,\infty,\rho^{-1}\right\}))\neq\emptyset$. By replacing $k$ with a field finitely generated over $k$, we may assume $\sqrt{\rho}\in k$.  Let $\alpha \in \text{Isom}^{\text{Dec}}_{G_{k}}(\Pi^{(m)}(\mathbb{P}^1_k-\left\{0,1,\infty,\rho\right\}),\Pi^{(m)}(\mathbb{P}^1_k-\left\{0,1,\infty,\rho^{-1}\right\}))$ satisying  $\Phi'(\alpha)=W$. Consider the cover  $\mathbb{P}^1_k-\left\{0,\pm 1,\infty,\pm\sqrt{\rho}\right\}\rightarrow \mathbb{P}^{1}_{k}-\left\{0,\infty,1,\rho\right\}$, $x\mapsto x^{2}$,  of degree $2$, and let  $H\subset \Pi^{(\pro 2,2)}(\mathbb{P}^{1}_{k}-\left\{0,\infty,1,\rho\right\})$ be the corresponding subgroup of index $2$ and $\tilde{H}:=\alpha(H)$. 
Then $\tilde{H}$ corresponds to a cover $\mathbb{P}^1_k-\left\{0,\pm\sqrt{a} ,\infty,\pm\sqrt{a\rho^{-1}}\right\}\rightarrow  \mathbb{P}^{1}_{k}-\left\{0,\infty,1,\rho^{-1}\right\}, x\mapsto\frac{1}{a}x^2,$ for some  $a\in k^{\times}$. By replacing $k$ with a field finitely generated over $k$, we may assume $\sqrt{a}\in k$ and then $a=1$ by coordinate transformation. Let  $V:=\mathbb{P}^1_k-\left\{0,\pm 1,\infty,\pm\sqrt{\rho}\right\},\tilde{V}:=\mathbb{P}^1_k-\left\{0,\pm 1,\infty,\pm\sqrt{\rho^{-1}}\right\}$, then we get
\begin{equation*}
\xymatrix@R=16pt{
\Pi^{(\pro 2,1)}(V)\ar@{}[d]|*=0[@]{\rotatebox{90}{$\cong$}}	&H\ar@{->>}[l]\ar@{^{(}-_>}[r] \ar@{}[d]|*=0[@]{\rotatebox{90}{$\cong$}}	^{\ \ \alpha_{\pro 2, 2}|_{H}}	&\Pi^{(\pro 2,2)}(\mathbb{P}^{1}_{k}-\left\{0,\infty,1,\rho\right\})\ar@{}[d]|*=0[@]{\rotatebox{90}{$\cong$}}	^{\ \ \alpha_{\pro 2, 2}} \\
\Pi^{(\pro 2,1)}(\tilde{V})&\tilde{H}\ar@{->>}[l]\ar@{^{(}-_>}[r] 	&\Pi^{(\pro 2,2)}(\mathbb{P}^{1}_{k}-\left\{0,\infty,1,\rho^{-1}\right\}),
}
\end{equation*}
where $\alpha_{\text{pro-}2,2}$ the image of $\alpha$ in $\text{Isom}^{\text{Dec}}_{G_{k}}(\Pi^{(\text{pro-}2,2)}(\mathbb{P}^1_k-\left\{0,1,\infty,\rho\right\}),\Pi^{(\text{pro-}2,2)}(\mathbb{P}^1_k-\left\{0,1,\infty,\rho^{-1}\right\}))$. Write $\beta$ for the isomorphism   $\Pi^{(\pro 2,1)}(V)\xrightarrow[G_k]{\sim}\Pi^{(\pro 2,1)}(\tilde{V})$. By  $\Phi'(\alpha)=W$, we have that  $\alpha_{\pro 2, 2}(\mathcal{I}_{0})=\mathcal{I}_{0}$, $\alpha_{\pro 2, 2}(\mathcal{I}_{\infty})=\mathcal{I}_{\infty}$, $\alpha_{\pro 2, 2}(\mathcal{I}_{1})=\mathcal{I}_{1}$. Thus,  $\beta$ also preserves the decomposition groups at cusps and
\[
\beta(\mathcal{I}_{0})=\mathcal{I}_{0},\ \ \beta(\mathcal{I}_{\infty})=\mathcal{I}_{\infty},\ \ \left\{\beta(\mathcal{I}_{1}),\beta(\mathcal{I}_{-1})\right\}=\left\{\mathcal{I}_{1},\mathcal{I}_{-1}\right\}\text{ and }\left\{\beta(\mathcal{I}_{\sqrt{\rho}}),\beta(\mathcal{I}_{-\sqrt{\rho}})\right\}=\left\{\mathcal{I}_{\sqrt{\rho^{-1}}},\mathcal{I}_{-\sqrt{\rho^{-1}}}\right\}
.\]
Considering the  $(-1)$-multiplication  if necessary, we may assume $\beta(\mathcal{I}_{1})=\mathcal{I}_{1}$. Dividing $\Pi^{(\pro 2,1)}(V)$ and $\Pi^{(\pro 2,1)}(\tilde{V})$ by $\langle\mathcal{I}_{-1},\mathcal{I}_{-\sqrt{\rho}}\rangle$ and $\langle\beta(\mathcal{I}_{-1}),\beta(\mathcal{I}_{-\sqrt{\rho}})\rangle$, respectively, we get an isomorphism:
\[
\overline{\beta}:\Pi^{(\pro 2, 1)}(\mathbb{P}^1_k-\left\{0,\infty,1,\sqrt{\rho}\right\}))\xrightarrow[G_k]{\sim}\Pi^{(\pro 2, 1)}(\mathbb{P}^1_k-\left\{0,\infty,1,u\cdot \sqrt{\rho^{-1}}\right\}) \ \ \ \ (u= 1\text{ or }-1)
\] 
which preserves the decomposition groups at cusps and such that $\overline{\beta}(\mathcal{I}_{0})=\mathcal{I}_{0}$, $\overline{\beta}(\mathcal{I}_{\infty})=\mathcal{I}_{\infty}$, $\overline{\beta}(\mathcal{I}_{1})=\mathcal{I}_{1}$. From the same argument as in the proof of Proposition \ref{2.2.2} (by using  Proposition \ref{2.1.3} and Proposition \ref{2.1.5}),  $\text{there exists }N\in \mathbb{N}\text{ with } 2\nmid N\text{ such that }\langle 1-\sqrt{\rho}\rangle^{N} \subset \langle 1-u\cdot  \sqrt{\rho^{-1}}\rangle $. Since $ 1-\sqrt{\rho}$ is a torsion-free element and 
\[ 
(1-\sqrt{\rho})(1-\sqrt{\rho^{-1}})^{-1}=-\sqrt{\rho},\ \ \ \ \ (1-\sqrt{\rho})(1+\sqrt{\rho^{-1}})=-\sqrt{\rho}^{3},
\]
we get  $\left[\langle 1-\sqrt{\rho}\rangle : \langle 1-\sqrt{\rho}\rangle \cap \langle 1-\sqrt{\rho^{-1}}\rangle \right]=12$ and $\left[\langle 1-\sqrt{\rho}\rangle : \langle 1-\sqrt{\rho}\rangle \cap \langle 1+\sqrt{\rho^{-1}}\rangle \right]=4$.  In both cases,  $N$ has to be divided by $4$. This is absurd. 
\end{proof}

\begin{prop}\label{2.2.4}
Assume that  $m\geq 3$, $ k$ is a field finitely generated over  $\mathbb{Q}$,  and that  $\text{prime}(C)$ coincides with the set of all primes. Let   $E_{1}$,$E_{2}$ finite sets of $k$-rational points of $\mathbb{P}^1_{k}$ with $|E_{1}|\geq 3$. Then, for $\alpha\in\text{Isom}_{G_{k}}(\Pi^{(m)}(\mathbb{P}^1_k-E_{1}),\Pi^{(m)}(\mathbb{P}^1_k-E_{2}))$,  there exists $f_{\alpha}\in \text{Aut}_{k}(\mathbb{P}^1_k)$ such that $f_{\alpha}(E_{1})=E_{2}$ and $f_{\alpha}\mid_{E_{1}}=\Phi(\alpha)$ hold, where $\Phi$ is the map (\ref{wq1.4.9}).
\end{prop}

\begin{proof}
Let  $\alpha\in \text{Isom}_{G_{k}}(\Pi^{(m)}(\mathbb{P}^1_k-E_{1}),\Pi^{(m)}(\mathbb{P}^1_k-E_{2}))$.   Since Aut($\mathbb{P}^1_{k}$) acts on $\mathbb{P}^1(k)$ triply transitively, we may assume that $E_{1}=\left\{0,\infty,1, \lambda_{1,1},\cdots,\lambda_{1,e}\right\}$, $E_{2}=\{0,\infty,1,\lambda_{2,1},\cdots, \lambda_{2,e}\}$, $\Phi(\alpha)(0)=0$, $\Phi(\alpha)(\infty)=\infty$, $\Phi(\alpha)(1)=1$, and $ \Phi(\alpha)(\lambda_{1,j})=\lambda_{2,j}$ for every $j\in \{1, \cdots,e\}$.\par
Let  $\alpha_{1}$ be the image of $\alpha$ in $\text{Isom}_{G_{k}}(\Pi^{(1)}(\mathbb{P}^1_k-E_{1}),\Pi^{(1)}(\mathbb{P}^1_k-E_{2}))$ and  $j\in\{1, \cdots,e\}$.  By Corollary \ref{1.4.5}, we have that $\alpha_{1}\in \text{Isom}^{\text{Dec}}_{G_{k}}(\Pi^{(1)}(\mathbb{P}^1_k-E_{1}),\Pi^{(1)}(\mathbb{P}^1_k-E_{2}))$.  Dividing  $\Pi^{(1)}(\mathbb{P}^1_k-E_{i})$ by $\langle\mathcal{I}_{\lambda_{i,h}}\mid1\leq h\leq e,j\neq h\rangle$, we get an isomorphism 
\[
\overline{\alpha}_{1}\in \text{Isom}_{G_{k}}^{\text{Dec}}(\Pi^{(1)}(\mathbb{P}^{1}_{k}-\left\{0,\infty,1,\lambda_{1,j}\right\}), \Pi^{(1)}(\mathbb{P}^{1}_{k}-\left\{0,\infty,1,\lambda_{2,j}\right\})
\]
 which satisfies that $\Phi'_{D}(\overline{\alpha}_{1})(0)=0$, $\Phi'_{D}(\overline{\alpha}_{1})(\infty)=\infty$, $\Phi'_{D}(\overline{\alpha}_{1})(1)=1$, and that  $\Phi'_{D}(\overline{\alpha}_{1})(\lambda_{1,j})=\lambda_{2,j}$.  Hence  we obtain  that $\lambda_{1,j}=\lambda_{2,j}$ or $\left\{\lambda_{1,j},\lambda_{2,j}\right\}=\left\{\rho,\rho^{-1}\right\}$ by Lemma \ref{2.2.3}(1). \par
 Let  $\alpha_{\text{pro-}2,2}$ be the image of $\alpha$ in $\text{Isom}_{G_{k}}(\Pi^{(\text{pro-}2,2)}(\mathbb{P}^1_k-E_{1}),\Pi^{(\text{pro-}2,2)}(\mathbb{P}^1_k-E_{2}))$.  By Corollary \ref{1.4.5}, we have that $\alpha_{\text{pro-}2,2}\in \text{Isom}^{\text{Dec}}_{G_{k}}(\Pi^{(\text{pro-}2,2)}(\mathbb{P}^1_k-E_{1}),\Pi^{(\text{pro-}2,2)}(\mathbb{P}^1_k-E_{2}))$.  Dividing  $\Pi^{(\text{pro-}2,2)}(\mathbb{P}^1_k-E_{i})$ by $\langle\mathcal{I}_{\lambda_{i,h}}\mid1\leq h\leq e,j\neq h\rangle$, we get an isomorphism 
\[
\overline{\alpha}_{\text{pro-}2,2}\in \text{Isom}_{G_{k}}^{\text{Dec}}(\Pi^{(\text{pro-}2,2)}(\mathbb{P}^{1}_{k}-\left\{0,\infty,1,\lambda_{1,j}\right\}), \Pi^{(\text{pro-}2,2)}(\mathbb{P}^{1}_{k}-\left\{0,\infty,1,\lambda_{2,j}\right\})
\]
 which satisfies that $\Phi'_{D}(\overline{\alpha}_{\text{pro-}2,2})(0)=0$, $\Phi'_{D}(\overline{\alpha}_{\text{pro-}2,2})(\infty)=\infty$, $\Phi'_{D}(\overline{\alpha}_{\text{pro-}2,2})(1)=1$, and that  $\Phi'_{D}(\overline{\alpha}_{\text{pro-}2,2})(\lambda_{1,j})=\lambda_{2,j}$.  Hence  we get  $\left\{\lambda_{1,j},\lambda_{2,j}\right\}\neq \left\{\rho,\rho^{-1}\right\}$ by Lemma \ref{2.2.3}. Thus,  $\lambda_{1,j}=\lambda_{2,j}$. As $j$ is arbitrary, the assertion follows.
\end{proof}

%%%%%%%%%%%%%%%%%%%%%%%%%%%%%%%%%%%%%%%%%%%%%%%%%%%%%%%%%%%%%%%%%%%%%%%%%%%%%%%%%%%

\subsection{Case of punctured projective lines over fields of positive characteristic}\label{chp}

\hspace{\parindent}In this subsection, we show the  $m$-step solvable Grothendieck conjecture for punctured projective lines in positive characteristic.  In characteristic $0$, we reduced the problem to the case of  the projective line minus $4$ points. In the positive characteristic, we also  approach the problem in a similarly way.  However, there are a problem in this way, which  do not exist in the case of characteristic $0$. 

\begin{definition}\label{2.3.1}
Let $S$ be a scheme over $\mathbb{F}_p$. we define the absolute Frobenius morphism $F_S:S\rightarrow S$ as the identity map on the underlying topological space and the $p$-power endomorphism on the structure sheaf.  Let $X$ be a scheme over $S$. We consider the following commutative diagram.
\begin{equation}\label{eq3.4}
\vcenter{
\xymatrix@C=46pt{
X\ar@/^10pt/[rrd]^{F_X}\ar@/_10pt/[rdd]	\ar@{.>}[rd]|{F_{X/S}}&&\\
&	X(1) \ar@{.>}[r] \ar@{.>}[d]&X\ar[d]\\
&	S\ar[r]^{F_S}&S
}
}
\end{equation}
Here, we set  $X(1):=X\times_{S,F_S} S$,  and   call it the  Frobenius twist of $X$ over $S$. For any $n\in\mathbb{N}\cup \left\{0\right\}$, We define the $n$-th Frobenius twist of $X$ inductively  by $X(0):=X,\ X(n):=X(n-1)(1)$. 
The morphism $F_{X/S}:X\rightarrow X(1)$ induced by the universality of the fiber product  is called the relative Frobenius morphism of $X$ over $S$.
\end{definition}

\begin{rem}\label{rem2.3.2}
Assume that $p>0$. Let $X$ be a scheme over $\text{Spec}(k)$. In general, $X$ and $X(1)$ are not isomorphic over $\text{Spec}(k)$. In particular, $F_{X/k}$ is not an isomorphism. For example, if $X=\mathbb{P}^1_{k}-\left\{0,\infty,1,\lambda\right\}$, $X$ and  $X(1)\underset{k}\cong\mathbb{P}^1_{k}-\left\{0,\infty,1,\lambda^{p}\right\}$ may not be isomorphic over $\text{Spec}(k)$. However,  the relative Frobenius morphism is a universal homeomorphism and the absolute  Frobenius morphism induces the  identity on the fundamental group. Thus, $\pi_{1}(F_{X/k}):\pi_1(X)\rightarrow\pi_1(X(1))$ is an isomorphism over $G_k$.
\end{rem}

\begin{definition}
Assume that $p>0$.  Let $k_{0}:=k\cap \overline{\mathbb{F}}_p$. A  curve $X$ over $k$ is isotrivial  if there exists a curve $X_{0}$ over $\overline{k}_0$ such thtat  $X_{0}\times_{\overline{k}_{0}}\overline{k}\underset{\overline{k}}\cong X_{\overline{k}}$.
\end{definition}

First, we reconstruct a given non-torsion element $\lambda$ of $k^{\times}$ from $\langle\lambda\rangle$ and $\langle1-\lambda\rangle$ in $k^{\times}$ when $p>0$. 

\begin{lem}\label{3.3.2}
Assume that $p>0$. Let $k_{0}:=k\cap \overline{\mathbb{F}}_p$. Let  $\lambda_{1}\in k-k_0$, $\lambda_{2}\in k^{\times}-\left\{1\right\}$ and $ u,v\in \mathbb{Z}$. If $\langle\lambda_{1}\rangle^{p^{u}}=\langle \lambda_{2}\rangle $ and $\langle1-\lambda_{1}\rangle^{p^{v}}=\langle 1-\lambda_{2}\rangle$ in $k^{\times}$, then there exists unique $n\in\mathbb{Z}$ such that $\lambda_{2}=\lambda_{1}^{p^n}$. 
\end{lem}

\begin{proof}
By assumption, $\lambda_{1}$ is a non-torsion element of $k^{\times}$. Then  $1-\lambda_{1}$ is a non-torsion element. $\langle\lambda\rangle^{p^{u}}=\langle \lambda_{2}\rangle $ and $\langle1-\lambda_{1}\rangle^{p^{v}}=\langle 1-\lambda_{2}\rangle $ implies that $\lambda_{2}$ and $1-\lambda_{2}$ are also non-torsion elements.  Hence we have $\mathbb{Z}\cong \langle\lambda\rangle^{p^{u}}=\langle \lambda_{2}\rangle $ and $\mathbb{Z}\cong \langle1-\lambda_{1}\rangle^{p^{v}}=\langle 1-\lambda_{2}\rangle$.  These imply  either $\lambda_{1}^{p^u}=\lambda_{2}$ or $\lambda_{1}^{p^u}=\lambda_{2}^{-1}$, and either $1-\lambda_{1}^{p^v}=1-\lambda_{2}$ or $1-\lambda_{1}^{p^v}=1-\lambda_{2}^{-1}$.  The assertion holds if $\lambda_{1}^{p^u}=\lambda_{2}$ or $1-\lambda_{1}^{p^v}=1-\lambda_{2}$. Thus, we may assume   that $\lambda_{1}^{p^u}=\lambda_{2}^{-1}$ and $1-\lambda_{1}^{p^v}=(1-\lambda_{2})^{-1}$.  Hence $\lambda_{1}$ satisfies $\lambda_{1}^{p^{u+v}}-\lambda_{1}^{p^{v}}+1=0$.  Let $W\in\mathbb{N}$ satisfy  $u+v+W\geq 0$ and $v+W\geq 0$. Then  $\lambda_{1}$ is a root of the polynomial $t^{p^{u+v+W}}-t^{p^{v+W}}+1\in \mathbb{F}_p[t]$. Hence we get $\lambda_{1}\in k_0$. This is absurd.
\end{proof}

\begin{lem}\label{2.3.5}
Let  $p$ be a prime number and  $X_{1},X_{2},Y_{1},Y_{2}\in\mathbb{Z}-\left\{0\right\}$.  Assume that 
\begin{equation}\label{eq2.3.5.1}
(p^{X_{1}}-1)(p^{Y_{1}}-1)=(p^{X_{2}}-1)(p^{Y_{2}}-1)\ \ (\text{in}\ \mathbb{Q})
\end{equation}
Then $ \left\{X_{1},Y_{1}\right\}=\left\{X_{2},Y_{2}\right\}$ holds.
\end{lem}

\begin{proof}

(Step\ $0$)\ Let $T\in\mathbb{Z}-\left\{0\right\}$. Then  $p^{T}-1>0\Leftrightarrow T>0$. Hence we get $X_{1}Y_{1}>0\iff X_{2}Y_{2}>0$ by assumption.  The following hold.
\[
p^{|T|}-1=
\begin{cases}
\ \ \ \ \ \ \ p^{0} (p^{T}-1)=\ \ p^{-\text{ord}_p(p^{T}-1)}(p^{T}-1)& (T>0)\\
\ \ -p^{-T}(p^{T}-1)=-p^{-\text{ord}_p(p^{T}-1)}(p^{T}-1)\ \ \ \ & (T<0)
\end{cases}
\]
Multiplying the absolute values of the both  sides of  (\ref{eq2.3.5.1}) by $p^{-\text{ord}_p(p^{X_{1}}-1)(p^{Y_{1}}-1)(p^{X_{2}}-1)(p^{Y_{2}}-1)}$, we get 
\begin{equation}\label{eq2.3.5.2}
 p^{-\text{ord}_p(p^{X_{2}}-1)(p^{Y_{2}}-1)}(p^{|X_{1}|}-1)(p^{|Y_{1}|}-1)=p^{-\text{ord}_p(p^{X_{1}}-1)(p^{Y_{1}}-1)}(p^{|X_{2}|}-1)(p^{|Y_{2}|}-1) \ \ \ (\text{in } \mathbb{Q} )
\end{equation}
Since $(p^{|X_{1}|}-1)(p^{|Y_{1}|}-1)$ and $(p^{|X_{2}|}-1)(p^{|Y_{2}|}-1)$ are not divided by $p$, (\ref{eq2.3.5.2}) implies $\text{ord}_p(p^{X_{1}}-1)(p^{Y_{1}}-1)=\text{ord}_p(p^{X_{2}}-1)(p^{Y_{2}}-1)\text{ and }(p^{|X_{1}|}-1)(p^{|Y_{1}|}-1)=(p^{|X_{2}|}-1)(p^{|Y_{2}|}-1)$. \\\ \\
(Step $1$) First, we consider the case  that $X_{1}>0$ and $Y_{1}>0$. By  symmetry, we may assume that $X_{1}\leq Y_{1}$ and $X_{2}\leq\tilde{Y_{1}}$.  (\ref{eq2.3.5.1}) implies 
\begin{equation}\label{eq2.3.5.3}
p^{X_{1}}(p^{Y_{1}}-p^{Y_{1}-X_{1}}-1)=p^{X_{2}}(p^{Y_{2}}-p^{Y_{2}-X_{2}}-1).
\end{equation}
$p^{Y_{1}}-p^{Y_{1}-X_{1}}-1$ is divided by $p$ if and only if $p=2$ and $X_{1}=Y_{1}$.
\setlength{\leftmargini}{10pt}  
\begin{itemize}
\item \underline{If $p\neq2$ or $(X_{1}\neq Y_{1}$ and $X_{2}\neq Y_{2})$.}\par
Since $p^{Y_{1}}-p^{Y_{1}-X_{1}}-1$ and $p^{Y_{2}}-p^{Y_{2}-X_{2}}-1$ are not divided by $p$, we get $X_{1}=\text{ord}_{p}(p^{X_{1}}(p^{Y_{1}}-p^{Y_{1}-X_{1}}-1))=\text{ord}_{p}(p^{X_{2}}(p^{Y_{2}}-p^{Y_{2}-X_{2}}-1))=X_{2}$. Then  (\ref{eq2.3.5.1}) implies $Y_{1}=Y_{2}$. Hence $\left\{X_{1},Y_{1}\right\}=\left\{X_{2},Y_{2}\right\}$ holds in this case.
\item \underline{If $p=2$ and $(X_{1}=Y_{1}$ or $X_{2}=Y_{2})$.}\par
We may assume that $X_{1}=Y_{1}$ by symmetry. If  $X_{2}=Y_{2}$, then (\ref{eq2.3.5.1}) implies $\left\{X_{1},Y_{1}\right\}=\left\{X_{2},Y_{2}\right\}$. Thus, we may assume  $X_{2}\neq Y_{2}$. (\ref{eq2.3.5.3}) implies $2^{X_{1}+1}(2^{X_{1}-1}-1)=2^{X_{2}}(2^{Y_{2}}-2^{Y_{2}-X_{2}}-1)$. By $X_{2}\neq Y_{2}$, the both sides of this equality are not $0$. Since $2^{X_{1}-1}-1$ $(\neq0)$ and $2^{Y_{2}}-2^{Y_{2}-X_{2}}-1$ are not divided by  $2$, we get $X_{1}+1=X_{2}$. Dividing the both sides by $2^{X_{1}+1}=2^{X_{2}}$, we have $2^{X_{1}-1}=2^{Y_{2}-X_{2}}(2^{X_{2}}-1)$. This implies $X_{2}=1$ and then $X_{1}=0$ by $X_{1}+1=X_{2}$. This is absurd.
\end{itemize}
Thus, the assertion holds if  $X_{1}>0$ and $Y_{1}>0$. 
\\\ \\
(Step $2$) Next, we consider the case that $X_{1}<0$ and $Y_{1}<0$.  Then we have $X_{2}Y_{2}>0$ by Step $0$.\par
If $X_{2}>0$ and $Y_{2}>0$, then  we get $X_{1}+Y_{1}=\text{ord}_p(p^{X_{1}}-1)(p^{Y_{1}}-1)=\text{ord}_p(p^{X_{2}}-1)(p^{Y_{2}}-1)=0$ by Step $0$. Since $X_{1}<0$ and $Y_{1}<0$, this is absurd.\par
If $X_{2}<0$ and $Y_{2}<0$, then we get $(p^{-X_{1}}-1)(p^{-Y_{1}}-1)=(p^{-X_{2}}-1)(p^{-Y_{2}}-1)$ by Step $0$. Since $-X_{1}$, $-Y_{1}$, $-X_{2}$, and $-Y_{2}$ are positive integers, Step $1$ implies $\left\{X_{1},Y_{1}\right\}=\left\{X_{2},Y_{2}\right\}$. Thus, the assertion holds if  $X_{1}<0$ and $Y_{1}<0$. 
\\\ \\
(Step $3$) Finally, we consider the case $X_{1}Y_{1}<0$. We may only consider the case $X_{1}>0$ and $Y_{1}<0$ by symmetry. Then we have $X_{2}Y_{2}<0$ by Step $0$. We may assume that $X_{2}>0$ and $Y_{2}<0$ by symmetry.\par
By Step $0$, we get $(p^{X_{1}}-1)(p^{-Y_{1}}-1)=(p^{X_{2}}-1)(p^{-Y_{2}}-1)$ and $Y_{1}=\text{ord}_p(p^{X_{1}}-1)(p^{Y_{1}}-1)=\text{ord}_p(p^{X_{2}}-1)(p^{Y_{2}}-1)=Y_{2}$. Since $X_{1}$, $-Y_{1}$, $X_{2}$, and $-Y_{2}$ are positive integers, Step $1$ implies  $\left\{X_{1},-Y_{1}\right\}=\left\{X_{2},-Y_{2}\right\}$. Hence we get $\left\{X_{1},Y_{1}\right\}=\left\{X_{2},Y_{2}\right\}$. Thus, the assertion follows.
\end{proof}

\begin{lem}\label{2.3.6}
Assume that $ k$ is a field finitely generated over  $\mathbb{F}_p$ and set $k_0:=\overline{\mathbb{F}_p}\cap k$. Let $\lambda_1\in k-k_0$, $\lambda_2\in k-\left\{0,1\right\}$ and $A_1,A_2,B_1,B_2,C_1,C_2\in\mathbb{Z}$. Assume that  
\begin{equation}\label{eq2.3.6-1}
A_1-A_2=B_1-B_2=C_1-C_2\text{ and} 
\end{equation}
\begin{equation}\label{eq2.3.6-2}
\lambda_1^{p^{A_1}-1}=\lambda_2^{p^{A_2}-1},\ \ \ \ (\lambda_1-1)^{p^{B_1}-1}=(\lambda_2-1)^{p^{B_2}-1},\ \ \ \ 
\left(\frac{\lambda_1}{\lambda_1-1}\right)^{p^{C_1}-1}=\left(\frac{\lambda_2}{\lambda_2-1}\right)^{p^{C_2}-1}.
\end{equation}
Then either of the following holds.
\begin{enumerate}[(a)]
\item $A_1=A_2, B_1=B_2 \text{ and }C_1=C_2$
\item $A_1=B_1=C_1 \text{ and }A_2=B_2=C_2$
\end{enumerate}
\end{lem}
\begin{proof}
$\lambda_2$ is a non-torsion element since  $\lambda_1^{p^{A_1}-1}=\lambda_2^{p^{A_2}-1}$ and $\lambda_1\in k-k_0$.\par
If $A_1=0$, then $A_2=0$ by  $\lambda_2^{p^{A_2}-1}=\lambda_1^{p^{0}-1}=1$. Hence  $A_1=A_2$ and (a) holds by (\ref{eq2.3.6-1}). Similarly, if either $A_2,B_1,B_2,C_1$ or $C_2$ is equal to $0$, then (a) holds by (\ref{eq2.3.6-1}). Hence we may assume that $A_1,A_2,B_1,B_2,C_1$ or $C_2$ is not equal to $0$.\par
By taking the $(p^{A_{2}}-1)(p^{B_{2}}-1)$-th power of  the third equality of  (\ref{eq2.3.6-2}), we get:
\begin{eqnarray}\label{eq2.3.6-3}
\lambda_1^{(p^{A_2}-1)(p^{B_2}-1)(p^{C_1}-1)}(\lambda_1-1)^{(p^{A_2}-1)(p^{B_1}-1)(p^{C_2}-1)} \nonumber\\
=\lambda_1^{(p^{A_1}-1)(p^{B_2}-1)(p^{C_2}-1)}(\lambda_1-1)^{(p^{A_2}-1)(p^{B_2}-1)(p^{C_1}-1)}.
\end{eqnarray}\par
As  $\lambda_1\in k-k_0$, $k_0[\lambda_1]$ is the polynomial ring in $\lambda_1$ with coefficients in  $k_0$. Accordingly  $\lambda_{1}$ and $\lambda_{1}-1$ are $\mathbb{Z}$-linear independent in $k_{0}(\lambda_{1})^{\times}$, in other words, $\langle \lambda_1,\lambda_1-1\rangle\cong\mathbb{Z}\oplus\mathbb{Z}$ holds in $k^{\times}$. Note that $p^{u}-1$ is contained in $\mathbb{Z}[1/p]$ for any $u\in \mathbb{Z}$. Since $\lambda_{1}$ and $\lambda_{1}-1$ are  $\mathbb{Z}$-linear independent in $k_{0}(\lambda_{1})^{\times}$,  (\ref{eq2.3.6-3}) implies $(p^{A_2}-1)(p^{C_1}-1)=(p^{A_1}-1)(p^{C_2}-1)$ and $(p^{B_2}-1)(p^{C_1}-1)=(p^{B_2}-1)(p^{C_1}-1)$. Thus, we get $\left\{A_2,C_1\right\}=\left\{A_1,C_2\right\}$ and $\left\{B_2,C_1\right\}=\left\{B_1,C_2\right\}$ by Lemma \ref{2.3.5}. By the first equality, we have $A_2=A_1$ or $A_2=C_2$. By the second  equality, we have $B_2=B_1$ or $B_2=C_2$.\par
If $A_2=A_1$, then  (a) holds by (\ref{eq2.3.6-1}). If $A_2=C_2$, then $A_1=C_1$ by $\left\{A_2,C_1\right\}=\left\{A_1,C_2\right\}$. If, moreover,  $B_2=B_1$, then (a) holds by  (\ref{eq2.3.6-1}). Otherwise, $B_2=C_2$, then $B_1=C_1$ by $\left\{B_2,C_1\right\}=\left\{B_1,C_2\right\}$. Therefore, (b) follows.
\end{proof}

Assume that $k$ be a field finitely generated over $\mathbb{F}_{p}$. Let $E_{1}$, $E_{2}$ be finite sets of $k$-rational points of $\mathbb{P}^1_{k}$ with $|E_{1}|\geq 3$, and $T_{i}\subset E_{i}$. Let  $\alpha\in \text{Isom}^{\text{Iner}}_{G_{k}}(\Pi^{(m)}(\mathbb{P}^1_{k}-E_{1}),\Pi^{(m)}(\mathbb{P}^1_{k}-E_{2}))$, $w_{1},w_{2}\in\mathbb{N}\cup\{0\}$, and $\alpha^{(w_{1},w_{2})}\in \text{Isom}^{\text{Iner}}_{G_{k}}(\Pi^{(m)}(\mathbb{P}^1_{k}-E_{1}(w_{1})),\Pi^{(m)}(\mathbb{P}^1_{k}-E_{2}(w_{2})))$ the isomorphism induced by $\alpha$.  Let  $f\in\text{Aut}_{k}(\mathbb{P}^{1}_{k})$. If $f(T_{1}(w_{1}))=T_{2}(w_{2})$\, then we write 
\[
\Phi'^{(w_{1},w_{2})}(\alpha):=\Phi'^{(w_{1},w_{2},T_{1},T_{2})}(\alpha)\in\text{Isom}_{\text{set}}(T_{1}(w_{1}),T_{2}(w_{2}))
\]
for the map defined as $\Phi'^{(w_{1},w_{2},T_{1},T_{2})}(\alpha)(x_{1}):=x_{2}$ when   $\alpha^{(w_{1},w_{2})}(\mathcal{I}_{x_{1},\Pi^{(m)}(\mathbb{P}^1_{k}-E_{1}(w_{1}))})=\mathcal{I}_{x_{2},\Pi^{(m)}(\mathbb{P}^1_{k}-E_{2}(w_{2}))})$ ($x_{i}\in T_{i}(w_{i})$).  We consider the following condition $(\dag)$ for  the tuple $(w_{1},w_{2},f, T_{1},T_{2},\alpha)$. 

\begin{itemize}
\item[(\dag)] $f(T_{1}(w_{1}))=T_{2}(w_{2})$ and $f|_{T_{1}(w_{1})}=\Phi'^{(w_{1},w_{2},T_{1},T_{2})}(\alpha)$.
\end{itemize}
Note that $\Phi'^{(0,0,T_{1},T_{2})}(\alpha)$ coincides with $\Phi'(\alpha)|_{T_{1}}$  (see (\ref{eqwr})).\par
By using these lemmas, we get the following result, which is a  positive characteristic version of Proposition \ref{2.2.4}. 

\begin{prop}\label{2.3.7}
Assume that $m\geq 3$, $ k$ is a field finitely generated over  $\mathbb{F}_p$, and that  $\text{prime}(C)$ coincides with the set of all primes that differ from $p$. Let $E_{1}$, $E_{2}$ be finite sets of $k$-rational points of $\mathbb{P}^1_{k}$ with $|E_{1}|\geq 3$ and assume that  the following condition holds. 
\begin{equation*}
(*)\ :\ \text{For all } S'\subset E_{1}\ \text{with}\ |S'|=4\ ,\ \mathbb{P}^1_{k}-S' \text{ is not isotrivial.}
\end{equation*}
Then, for $\alpha\in\text{Isom}_{G_{k}}(\Pi^{(m)}(\mathbb{P}^1_k-E_{1}),\Pi^{(m)}(\mathbb{P}^1_k-E_{2}))$,  there exists  $w_{1},w_{2}\in \mathbb{N}\cup \left\{0\right\}$ and $f\in \text{Aut}_{k}(\mathbb{P}^1_k)$ such that the condition ($\dag$) for $(w_{1},w_{2},f, E_{1},E_{2},\alpha_{1})$ holds, where $\alpha_{1}$ stands for  the element of  $\text{Isom}_{G_{k}}^{\text{Iner}}(\Pi^{(1)}(\mathbb{P}^1_k-E_{1}), \Pi^{(1)}(\mathbb{P}^1_k-E_{2}))$ induced by $\alpha$.
\end{prop}
\begin{proof}
We consider the following four steps.\\
(Step $1$) If  $|E_{1}|=3$, then the assertion clearly holds. We consider the case that  $|E_{1}|=4$.   Let  $x_1,x_2,x_3,x_4$ be distinct elements of $E_{1}$ and set  $\varepsilon=\left\{x_1,x_2\right\},\delta=\left\{x_3,x_4\right\}$.   Since  $\kappa_n(\varepsilon,\delta)$ is  determined  group-theoretically  by $\Pi^{(1)}(\mathbb{P}^1_k-E_{1})\twoheadrightarrow G_k$ and the decomposition groups of $\Pi^{(1)}(\mathbb{P}^1_k-S)$ at cusps, we get $k(\mu_n,\lambda(\varepsilon,\delta)^{\frac{1}{n}})=k(\mu_n,\lambda(\Phi'(\alpha)(\varepsilon),\Phi'(\alpha)(\delta))^{\frac{1}{n}})$ for all $n\in\mathbb{N}(C)$ by Proposition \ref{2.1.3}.  Hence there exists $u\in \mathbb{Z} $  such that $\langle\lambda(\varepsilon,\delta)\rangle^{p^{u}}=\langle\lambda(
\Phi'(\alpha)(\varepsilon),
\Phi'(\alpha)(\delta))\rangle$ by Corollary \ref{2.1.6}(2). Applying this argument to $\varepsilon'=\left\{x_3,x_2\right\},\delta'=\left\{x_1,x_4\right\}$, we get  $v\in\mathbb{Z}$ such that $\langle1-\lambda(\varepsilon,\delta)\rangle^{p^{v}}=\langle1-\lambda(
\Phi'(\alpha)(\varepsilon),
\Phi'(\alpha)(\delta))\rangle$, since $\lambda(\varepsilon',\delta')=1-\lambda(\varepsilon,\delta)$.  As $\lambda(\varepsilon,\delta)$ is not contained in $k_0:=k\cap\overline{\mathbb{F}}_p$ by the condition $(*)$, there  exists a unique $n\in\mathbb{Z}$ such that $\lambda(
\Phi'(\alpha)(\varepsilon),
\Phi'(\alpha)(\delta))=\lambda(\varepsilon,\delta)^{p^n}$ by Lemma \ref{3.3.2}.\par
Let $u\in\mathbb{Z}$. If $u\geq 0$, we define $(w_{1}(u),w_{2}(u))$ to be $(u,0)$. Otherwise, we define $(w_{1}(u),w_{2}(u))$ to be $(0,-u)$.\par
We may assume that  $E_{1}=\left\{0,\infty,1, \lambda_{1},\mu_{1},\cdots\right\}$, $E_{2}=\left\{0,\infty,1,\lambda_{2},\mu_{2}\cdots\right\}$ and $\Phi'(\alpha)(0)=0$, $\Phi'(\alpha)(\infty)=\infty$, $\Phi'(\alpha)(1)=1$, $ \Phi'(\alpha)(\lambda_{1})=\lambda_{2}$, $ \Phi'(\alpha)(\mu_{1})=\mu_{2},\cdots$  because Aut($\mathbb{P}^1_{k}$) acts on $\mathbb{P}^1(k)$ triply transitively. The condition ($*$) implies $E_{1}-\{0,1,\infty\}\subset  k-k_0$ and  $E_{2}-\{0,1,\infty\}\subset  k-k_0$. Since $\lambda(\left\{0,\infty\right\},\left\{1,\lambda_{1}\right\})=\lambda_{1}$, the  above argument implies that there exists a unique $n_1\in\mathbb{Z}$ such that $\lambda_{2}=\lambda_{1}^{p^{n_1}}$. Hence, ($\dag$) for $(w_{1}(n_{1}),w_{2}(n_{1}),\text{id}:\mathbb{P}^1_k(w_{1}(n_1))=\mathbb{P}^{1}_{k}(w_{2}(n_1)),E_{1},E_{2},\alpha_{1})$ follows. Thus, we may assume that $|E_{1}|\geq 5$. \\\ \\
(Step $2$) We consider the case that $|E_{1}|=5$.  We get the following consequence by considering various $\varepsilon,\delta$.
\[
\begin{cases}
\lambda(\left\{0,\infty\right\},\left\{1,\lambda_{1}\right\})=\lambda_{1}. \text{ Hence there exists a unique }n_{\lambda}\in\mathbb{Z} \text{ such that }\lambda_{2}=\lambda_{1}^{p^{n_{\lambda}}}.\\
\lambda(\left\{0,\infty\right\},\left\{1,\mu_{1}\right\})=\mu_{1}. \text{ Hence there exists a unique }n_{\mu}\in\mathbb{Z} \text{ such that }\mu_{2}=\mu_{1}^{p^{n_{\mu}}}.\\
\lambda(\left\{0,\infty\right\},\left\{\lambda_{1},\mu_{1}\right\})=\frac{\mu_{1}}{\lambda_{1}}.\text{ Hence there exists a unique }\sigma\in\mathbb{Z}\text{ such that }\frac{\mu_{2}}{\lambda_{2}}=(\frac{\mu_{1}}{\lambda_{1}})^{p^{\sigma}}.\\
\lambda(\left\{1,\infty\right\},\left\{\lambda_1,\mu_{1}\right\})=\frac{\mu_{1}-1}{\lambda_1-1}.\text{ Hence there exists a unique }\tau\in\mathbb{Z} \text{ such that }\frac{\mu_2-1}{\lambda_2-1}=(\frac{\mu_1-1}{\lambda_1-1})^{p^{\tau}}.\\
\lambda(\left\{0,1\right\},\left\{\lambda_1,\mu_{1}\right\})=\frac{\mu_1}{\mu_{1}-1}\frac{\lambda_1-1}{\lambda_1}.\text{ Hence there exists a unique }\zeta\in\mathbb{Z}\text{ such that }\frac{\mu_2}{\mu_2-1}\frac{\lambda_2-1}{\lambda_2}=(\frac{\mu_{1}}{\mu_{1}-1}\frac{\lambda_1-1}{\lambda_1})^{p^{\zeta}}.
\end{cases}
\]
Since $\left\{0,\infty,1,\lambda_1\right\},\left\{0,\infty,1,\mu_{1}\right\},\left\{0,\infty,\lambda_1,\mu_{1}\right\},\left\{1,\infty,\lambda_1,\mu_{1}\right\},\left\{0,1,\lambda_1,\mu_{1}\right\}\subset S$  intersect at three points with one another, we have only to show that at least two of $n_{\lambda}$, $n_{\mu}$, $\sigma$, $\tau$, $\zeta$ are equal by Step $1$.  We have 
\[
\lambda_2=\lambda_1^{p^{n_\lambda}},\ \mu_2=\mu_1^{p^{n_\mu}},\ \frac{\mu_{2}}{\lambda_{2}}=\left(\frac{\mu_{1}}{\lambda_{1}}\right)^{p^{\sigma}}\ \Longrightarrow\  \frac{\mu_{1}^{p^{n_\mu}}}{\lambda_1^{p^{n_\lambda}}}=\frac{\mu_{2}}{\lambda_2}=\left(\frac{\mu_{1}}{\lambda_1}\right)^{p^{\sigma}}\ \Longrightarrow\  \lambda_1^{p^{(n_\lambda-\sigma)}-1}=\mu_{1}^{p^{(n_\mu-\sigma)}-1}.
\]
Similarly,   we get the following equations.
\[
(\lambda_1-1)^{p^{(n_\lambda-\tau)}-1}=(\mu_{1}-1)^{p^{(n_\mu-\tau)}-1},\ \ \ \ \ 
\left(\frac{\lambda_1}{\lambda_1-1}\right)^{p^{(n_\lambda-\zeta)}-1}=\left(\frac{\mu_{1}}{\mu_{1}-1}\right)^{p^{(n_\mu-\zeta)}-1}.
\]
Hence  Lemma \ref{2.3.6} implies $(n_\lambda-\sigma=n_\mu-\sigma$, $n_\lambda-\tau=n_\mu-\tau $ and $n_\lambda-\zeta=n_\mu-\zeta)$ or $(n_\lambda-\sigma=n_\lambda-\tau=n_\lambda-\zeta $ and $n_\mu-\sigma=n_\mu-\tau=n_\mu-\zeta )$. Thus, we get $n_\lambda=n_\mu$ or $\sigma=\tau=\zeta$.  In the both cases, the assertion follows.\\\ \\
(Step $3$) Let $T_{1}$ and $T'_{1}$  be  subsets of $E_{1}$ that satisfy $|T_{1}\cap T'_{1}|\geq 3$. Set $T_{2}:=\Phi'(\alpha)(T_{1})$ and $T'_{2}:=\Phi'(\alpha)(T'_{1})$.  If  there exist  $w_{1},w_{2}\in\mathbb{Z}$ and $f,$ $f'\in\text{Aut}_{k}(\mathbb{P}^1_{k})$ suth that  ($\dag$) for $(w_{1},w_{2},f,T_{1},T_{2},\alpha_{1})$ and  ($\dag$) for $(w_{1},w_{2},f',T'_{1},T'_{2},\alpha_{1})$  hold, then  
\[
f\mid_{(T_{1}\cap T'_{1})(w_{1})}=\Phi'^{(w_{1},w_{2},T_{1},T_{2})}(\alpha)\mid_{(T_{1}\cap T'_{1})(w_{1})}=\Phi'^{(w_{1},w_{2},T'_{1},T'_{2})}(\alpha)\mid_{(T_{1}\cap T'_{1})(w_{1})}=f\mid_{(T_{1}\cap T'_{1})(w_{1})}.
\]
 Thus, we get $f=f'$ since  $|T_{1}\cap T'_{1}|\geq 3$. Therefore,  ($\dag$) for $(w_{1},w_{2},f,T_{1}\cup T'_{1},T_{2}\cup T'_{2},\alpha_{1})$ follows.\\\ \\
(Step $4$)  Finally, we consider the case $|S|>5$.   Let $x\in E_{1}-\{0,1,\infty\}$. Set $T_x:=\left\{0,\infty,1,x\right\}$. By Step $1$, there exists a unique $n_x\in\mathbb{Z}$ and a unique $f_{x}\in\text{Aut}(\mathbb{P}^1_{k})$ such that ($\dag$) for $(w_{1}(n_x),w_{2}(n_{x}),f_x,T_x,\Phi'(\alpha)(T_x),\alpha_{1})$ holds. Applying Step $2$ to $T_{x}\cup T_{y}$ ($y\in E_{1}-\{0,1,\infty\}$), we get $n_x=n_{y}$.  Hence  Step $3$ implies  $f_x=f_{y}$. Since $y$ is arbitrary, we obtain  ($\dag$) for $(w_{1}(n_x),w_{2}(n_x),f_x,E_{1},E_{2},\alpha_{1})$. Thus, the assertion follows.
\end{proof}

%%%%%%%%%%%%%%%%%%%%%%%%%%%%%%%%%%%%%%%%%%%%%%%%%%%%%%%%%%%%%%%%%%%%%%%%%%%%%%%%%%%

\subsection{Main theorem}\label{main}

\begin{theorem}\label{2.4.1}
Let $i=1,2$.  Let $X_{i}$ be a proper smooth curve over $k$  and $E_{i}$ a closed subscheme of $X_{i}$ which is finite \'{e}tale over $k$. Set $U_{i}:=X_{i}-E_{i}$.   Assume that $k$ is a field finitely generated over the prime field, $\text{prime}(C)$ coincides with the set of all primes that differ from $p$, $m\geq 3$, and that $U_{1}$ is  genus $0$ hyperbolic curve over $k$. If, moreover  $p>0$, we assume that 
\begin{equation*}
(**):\text{For each }S'\subset E_{1,\overline{k}}\ \text{with}\ |S'|=4,\text{ the curve }X_{1,\overline{k}}-S'\ \text{is\ not\  isotrivial.}
\end{equation*}
Then the following hold.
\[
\Pi^{(m)}(U_{1})\underset{G_k}{\cong} \Pi^{(m)}(U_{2})\Longrightarrow
\begin{cases}
U_{1}\underset{k}{\cong}U_{2} & (p=0) \\
U_{1}(n_{1})\underset{k}{\cong}U_{2}(n_{2})\text{ for some }n_{1},n_{2}\in \mathbb{N}\cup \left\{0\right\} & (p>0)
\end{cases}
\]
\end{theorem}

\begin{proof}
As an isomorphism $\Pi^{(m)}(U_{1})\underset{G_k}{\cong} \Pi^{(m)}(U_{2})$ induces an isomorphism $\Pi^{(3)}(U_{1})\underset{G_k}{\cong} \Pi^{(3)}(U_{2})$, we may assume  that $m=3$.  Let $g_{i}:=g(U_{i})$ be the  genus of $X_{i,\overline{k}}$ and $r_{i}:=r(U_{i}):=|E_{i,\overline{k}}|$.   We have that $g_{i}=\frac{1}{2}\text{rank}_{\mathbb{Z}_{\ell}}(\overline{\Pi}^{\ab ,\pro\ell}(U_{i})/W_{-2}(\overline{\Pi}^{\ab ,\pro\ell}(U_{i})))$. This  implies  $g_{2}=0$.   There exists a finite Galois extension $K$ of  $k$ that satisfies $U_{1,K}\cong\mathbb{P}^1_K-S_{1}$ and $U_{2,K}\cong\mathbb{P}^1_K-S_{2}$ for some $S_{1},S_{2}\subset\mathbb{P}^1_K(K)$, respectively. 
Consider  an isomorphism $\Pi^{(3)}( U_{1})\xrightarrow[G_k]{\sim} \Pi^{(3)}(U_{2})$. It induces  an isomorphism  $\alpha: \Pi^{(1)}(U_{1})\xrightarrow[G_k]{\sim} \Pi^{(1)}(U_{2})$  that preserves the decomposition groups by Corollary \ref{1.4.5}(1).  $\alpha$ induces $\alpha_{K}:\Pi^{(1)}( U_{1,K})\xrightarrow[G_K]{\sim} \Pi^{(1)}( U_{2,K})$.  The condition ($**$) for $ U_{1}$ in Theorem \ref{2.4.1} implies the condition ($*$) for $ U_{1,K}$ in Proposition \ref{2.3.7}. Hence the following hold by Proposition \ref{2.2.4} and Proposition \ref{2.3.7} 
\setlength{\leftmargini}{10pt}  
\begin{itemize}
\item Case $p=0$. There exists $f:  \mathbb{P}^{1}_{K}\xrightarrow[K]{\sim} \mathbb{P}^{1}_{K}$ such that $f(S_{1})=S_{2}$, and that $f(x_{1})=x_{2} \iff \alpha_{K}(\mathcal{I}_{x_{1}})=\mathcal{I}_{x_{2}}$ for every $x_{1}\in S_{1}$, $x_{2}\in S_{2}$
\item Case $p>0$. There exist $w_{1},w_{2}\in \mathbb{N}\cup \left\{0\right\}$ and $f:  \mathbb{P}^{1}_{K}\xrightarrow[K]{\sim}\mathbb{P}^{1}_{K}$ such that $f(S_{1}(w_{1}))=S_{2}(w_{2})$, and that $f(x_{1})=x_{2} \iff \alpha_{K}(\mathcal{I}_{x_{1}})=\mathcal{I}_{x_{2}}$  for every $x_{1}\in S_{1}(w_{1})$, $x_{2}\in S_{2}(w_{2})$.
\end{itemize}\noindent
We write  $\mathbb{U}_{i}$ instead of   $U_{i}$ (resp. $U_{i}(w_{i})$) for $i=1,2$ when $p=0$ (resp. $p>0$).  By  the above argument, we have  that $\mathbb{U}_{i,K}^{\text{cpt}}\xrightarrow[K]{\sim}\mathbb{P}^1_{K}$ and  $(\mathbb{U}_{i,K}^{\text{cpt}}-\mathbb{U}_{i,K})\subset\mathbb{P}^1_K(K)$. Moreover,  there exists $f:  \mathbb{P}^{1}_{K}\xrightarrow[K]{\sim} \mathbb{P}^{1}_{K}$ such that $f(\mathbb{U}_{1,K}^{\text{cpt}}-\mathbb{U}_{1,K})=(\mathbb{U}_{2,K}^{\text{cpt}}-\mathbb{U}_{2,K})$, and that  $f(x_{1})=x_{2} \iff \alpha_{K}(\mathcal{I}_{x_{1}})=\mathcal{I}_{x_{2}}$ for every $x_{1}\in (\mathbb{U}_{1,K}^{\text{cpt}}-\mathbb{U}_{1,K})$, $x_{2}\in (\mathbb{U}_{2,K}^{\text{cpt}}-\mathbb{U}_{2,K})$.\par
Let   $x_{i}\in (\mathbb{U}_{i,K}^{\text{cpt}}-\mathbb{U}_{i,K})$ and $\rho(\mathbb{U}_{i})$ the image of  $\rho\in \text{Gal}(K/k)$ by $\text{Gal}(K/k)\rightarrow \text{Aut}_{\mathbb{U}}(\mathbb{U}_{i,K}), \rho\mapsto id_{\mathbb{U}_{i}}\times\rho$. Let $\underline{\rho}(\mathbb{U}_{1})$ be an inverse image of $\rho$ by $\Pi^{(1)}(\mathbb{U}_{1})\twoheadrightarrow G_k$ and $\underline{\rho}(\mathbb{U}_{2})$ the image of $\underline{\rho}(\mathbb{U}_{1})$ by $\alpha$.
We have  that $\underline{\rho}(\mathbb{U}_{i})\cdot \mathcal{I}_{x_{i}}\cdot \underline{\rho}(\mathbb{U}_{i})^{-1}=\mathcal{I}_{\rho(\mathbb{U}_{i})(x_{i})}$ for   $i=1,2$.  This implies $\alpha(\mathcal{I}_{\rho(\mathbb{U}_{1})(x_{1})})=\underline{\rho}(\mathbb{U}_{2})\cdot \alpha(\mathcal{I}_{x_{1}})\cdot \underline{\rho}(\mathbb{U}_{2})^{-1}=\mathcal{I}_{\rho(\mathbb{U}_{2})(f(x_{1}))}$. Hence we get $f(\rho(\mathbb{U}_{2})(x_{1}))=\rho(\mathbb{U}_{2})(f(x_{1}))$. Since  $|\mathbb{U}_{1,\overline{k}}^{\text{cpt}}-\mathbb{U}_{1,\overline{k}})|\geq3$, this implies   that $ f\circ\rho(\mathbb{U}_{1})=\rho(\mathbb{U}_{2})\circ  f$ for all $\rho$. By Galois descent, we obtain that  $\mathbb{U}_{1}\underset{k}\cong\mathbb{U}_{2}$.
\end{proof}

\noindent
Research Institute for Mathematical Sciences\\
Kyoto University\\
Kyoto 606-8502\\
JAPAN\\
naganori@kurims.kyoto-u.ac.jp \\

%%%%%%%%%%%%%%%%%%%%%%%%%%%%%%%%%%%%%%%%%%%%%%%%%%%%%%%%%%%%%%%%%%%%%%%%%%%%%%%%%%%%%%%%%%%%%%%%%

\end{document}